\documentclass[11pt,a4paper,twoside]{amsart}
\usepackage{amsmath,amssymb,amsfonts,amsthm}
\usepackage{hyperref,color}
\usepackage{times}
\usepackage{algorithm}
\usepackage{algorithmic}

\newtheorem{The}{Theorem}[section]

\newtheorem{Theorem}{Theorem}[section]
\newtheorem{Proposition}[The]{Proposition}
\newtheorem{Lemma}[The]{Lemma}

\theoremstyle{definition}
\newtheorem{Definition}[The]{Definition}
\newtheorem{Remark}[The]{Remark}
\newtheorem{Example}[The]{Example}

\newcommand{\R}{\mathbb{R}}      
\newcommand{\C}{\mathbb{C}}

\def \ord{{\rm ord}}
\def \rank{{\rm rank}}
\def \ld{{\rm ld}}
\def \lc{{\rm lc}}
\def \lcd{{\rm lcd}}
\def \lcm{{\rm lcm}}
\def \I{{\rm in}}

\subjclass[2000]{68U05, 32M05, 32V40, 12H05}

\begin{document}

\title{
Applications of differential algebra
\\
for computing Lie algebras of
\\
infinitesimal CR-automorphisms}
\author{Masoud Sabzevari}
\address{Department of Mathematics,
Shahrekord University, 115 Shahrekord, IRAN}
\email{sabzevari@math.iut.ac.ir}

\author{Amir Hashemi}
\address{Department of Mathematical Sciences,
Isfahan University of Technology, Isfahan, IRAN}
\email{a.hashemi@cc.iut.ac.ir}

\author{Benyamin M.-Alizadeh}
\address{Department of Mathematical Sciences,
Isfahan University of Technology, Isfahan, IRAN \ \ and \ \ Young
Researchers Club, Science and Research Branch, Islamic Azad
University, Tehran, 461/15655, Iran.}
\email{benyamin.m.alizadeh@gmail.com}

\author{Jo\"{e}l Merker}
\address{D\'epartment de Math\'ematiques d'Orsay, B\^{a}timent 425,
Facult\'{e} des Sciences, Universit\'{e} Paris XI - Orsay, F-91405
Orsay Cedex,  FRANCE} \email{merker@dma.ens.fr}

\date{\number\year-\number\month-\number\day}

\maketitle

\begin{abstract}
We perform {\em detailed} computations of Lie algebras of
infinitesimal CR-automorphisms associated to three specific model
real analytic CR-generic submanifolds in $\C^9$ by employing
differential algebra computer tools\,\,---\,\,mostly within the {\sc
Maple} package {\tt DifferentialAlgebra}\,\,---\,\,in order to
automate the handling of the arising highly complex linear systems of
{\sc pde}'s. Before treating these new examples which prolong
previous works of Beloshapka, of Shananina and of Mamai, we provide
general formulas for the explicitation of the concerned {\sc pde}
systems that are valid in arbitrary codimension $k \geqslant 1$ and
in any CR dimension $n \geqslant 1$.  Also, we show how Ritt's
reduction algorithm can be adapted to the case under interest, where
the concerned {\sc pde} systems admit so-called {\sl complex
conjugations}.
\end{abstract}

\pagestyle{headings} \markright{Applications of Differential Algebra
to the computation of Lie algebras of infinitesimal CR-automorphisms}

\section{Introduction}
\label{Introduction}

The Lie algebras $\frak{aut}_{CR}\big( M_{\sf model} \big)$ of
infinitesimal CR-automorphisms of various {\em model}\,\,---\,\,in the
sense of Beloshapka\,\,---\,\,Cauchy-Riemann (CR) generic submanifolds
$M_{\sf model} \subset \C^{ n + k}$ of codimension $k \geqslant 1$ and
of CR dimension $n \geqslant 1$ are key algebraic features which open
the door to a potentially infinite number of new Cartan geometries.
Indeed, the knowledge of $\frak{ aut}_{ CR} \big( M_{\sf
model} \big)$ and of its isotropy subalgebras $\mathfrak{ aut}_{ CR}
\big( M_{\sf model}, p \big)$ at points $p \in M_{\sf model}$ strongly
intervenes when one endeavours to build Cartan connections associated
to {\em all} {\sl geometry-preserving} real analytic deformations $M
\subset \C^{ n+k}$ of a chosen model. It is well known that procedures
due to Cartan and to Tanaka exist to perform such constructions
({\em see}~\cite{Chern-Moser,
BES, EMS, AMS, Isaev, MS, 5-cubic}
in a CR context), although the practical outcome
appears most of the times to be delicate and unpredictable.

In addition, there has recently been an increasing interest towards
complete classification of CR-generic submanifolds according to their
algebras of infinitesimal CR-automorphisms. Notably, Beloshapka and
Kossovskiy \cite{Beloshapka2011} classified homogeneous CR-generic
submanifolds $M^4 \subset \C^3$ of CR dimension $1$, while a bit
before, Fels and Kaup \cite{Fels} classified the Levi-degenerate
homogeneous $2$-nondegenerate hypersurfaces $M^5 \subset \C^3$ of
dimension five.

Far beyond for what concerns the appearing (co)dimensions, Beloshapka
and his school in the last decade devised general procedures to cook
up {\sl model CR-generic submanifolds} that, most often, have rigid
polynomial defining equations. But when some concrete equation of a
CR-generic manifold is given, one unpleasant obstacle happens to be
the {\em complexity} and the {\em length} of the computations that are
required to attain the full Lie algebras of infinitesimal CR-automorphisms
({\em see}~\cite{AMS, Beloshapka1997, 5-cubic, MS, Shananina2000}),
an obstacle which motivates the present work.

By $n \geqslant 1$ and $k
\geqslant 1$ throughout this paper, we shall mean the CR dimension and
the codimension of a real analytic local CR-generic
submanifold $M^{ 2n + k}
\subset\mathbb C^{n+k}$ passing through a reference point,
say through the origin.
Beloshapka in \cite{Beloshapka2004}, introduced a significant class
of CR-generic manifolds with several {\em nice} properties ({\em see}
{\em e.g.} Theorem 14 in
\cite{Beloshapka2004}),
that he denoted by
$Q(n,k)\subset\mathbb C^{n+k}$ and called {\sl universal
models}. He also computed the algebras of infinitesimal
CR-automorphisms associated to the simplest model $Q(1,2)$ in
\cite{Beloshapka1997}, and he derived some interesting
stability results.
Subsequently, Shananina computed such algebras for the universal
models $Q(1,k)$ with $3\leqslant k \leqslant 7$ in \cite{Shananina2000} and
derived expected consequences
too ({\em see} Theorem 1, Propositions 1 and 2 and Corollary 1 of
\cite{Shananina2000}). Finally, in \cite{Mamai}, Mamai studied Lie
algebras of infinitesimal CR-automorphisms associated to some universal
models $Q(1,k)$ for $8\leqslant k\leqslant 12$, though without presenting
details. As far as the authors are aware of, in CR dimension $n = 1$,
no higher codimensions have been explored. Understandably, as much as
the dimension or codimension of a CR-generic submanifold $M^{ 2n + k}
\subset \C^{ n+k}$ increases, the size and the complexity of the
corresponding computations of $\frak{aut}_{CR}(M)$ growths rapidly,
hence an automation would be welcome, even a partial one.

Serendipitously, an important, recently renewed, much
related subject has been
extensively studied: {\sl Differential (Computer) Algebra}, {\em
cf.}  \cite{Boulier,Blop95,Gallo, Kolchin,Ritt}. There, one employs
algebraic tools\,\,---\,\,like the ones of Gr\"obner bases
theory\,\,---\,\,in order to solve systems of partial differential
equations, or in order to find the Lie symmetries of differential
equations, a vast area too. Over the past few years, several relative
packages have been developed within various computer
algebra
systems. For instance, two {\sc Maple} packages
were designed in this direction, firstly {\tt
DifferentialAlgebra} by Boulier, Lazard, Ollivier, Petitot
(\cite{Boulier}), and secondly {\tt diffgrob2} by Mansfield
(\cite{Mansfield1997}).

In this paper, we aim to provide an effective algorithm in order to compute
the Lie algebras $\frak{aut}_{CR}(M_{\sf model})$ associated to
model real analytic
CR-generic submanifolds $M_{\sf model}^{ 2n+k}
\subset \C^{ n+k}$, by the valuable means of differential
algebra, supplemented by some new operations. For this
purpose, we shall denote by ${\sf LinCons}$ the {\sc pde} systems that
have
constant (complex) coefficients,
precisely as the ones we shall encounter
several times.
Since these systems admit complex-valued equations, we equip at first
the fundamental {\em Ritt's reduction theorem} with an operator,
which we call the {\sl bar-reduction}, and we extend it to
gain the following
conclusion, more appropriate to treat the arising
complex-valued ${\sf LinCons}$ {\sc pde}
systems ({\em see} Theorem \ref{ModRitt-Thm}):

\begin{Theorem}
{\sc (Extended Ritt's reduction theorem)} Consider a differential
polynomial ring $R=\mathbb{C}\{u_1,\ldots,u_n\}$ over the field of
complex numbers, let $\Theta$ be the set of its derivation operators,
and let `$>$' be a ranking over $\Theta U$. Furthermore, assume that
$p\in R$ is a ${\sf LinCons}$ differential polynomial and let $Q$ be a
finite set of ${\sf LinCons}$ differential polynomials. Then, there
exists $r \in R$, and for each $q\in Q$, there exists $\theta_q,
\theta_{\overline{ q}} \in
\Theta$ and $c_q, c_{\overline{ q}}\in K$
satisfying the following conditions:

\begin{itemize}

\smallskip\item
$p= \big( \sum_{q\in Q}\,c_q\,\theta_q\, q \big) +
\big( \sum_{q\in
Q}\,c_{\overline{q}}\,\theta_{\overline{q}} \,\overline{q} \big) + r,$

\smallskip\item
for each $q$ appearing in this summation we have:
\[\rank(r)
<
{\rm min}_>
\big\{\rank(\theta_qq),
\rank(\theta_{\overline{q}}{\overline{q}})
\big\},
\]

\smallskip\item{no term of $r$ is a derivation of either $\rank(q)$ or
$\rank(\overline{q})$ for each $q \in Q$.}
\end{itemize}
\end{Theorem}

Using this result, we modify the {\em Rosenfeld-Gr\"obner} algorithm
to an algorithm we call the {\sl LRG} algorithm (Algorithm \ref{LRG}
below) in such a way that it works appropriately in the case of our
${\sf LinCons}$ complex-valued {\sc pde} systems. It is noteworthy
that, for the sake of simplicity, we study only {\em rigid} CR-generic
submanifolds $M^{ 10} \subset \mathbb{C}^{1+8}$ ({\em see} Subsection
\ref{mainstrategy} for more details), because the arising {\sc pde}
systems {\em are} ${\sf LinCons}$; of course, a more general algorithm
could be described for non-rigid $M_{\sf model}^{ 2n+k} \subset \C^{
n+k}$, the {\sc pde} system still being linear, though with
non-necessarily constant coefficients.

Specifically, we employ our {\sl LRG} algorithm for computing {\em
explicitly} the Lie algebras of infinitesimal CR-automorphisms
associated to the following three real analytic rigid CR-generic
submanifolds $M^{ 10} \subset \mathbb{ C}^9$, represented in
coordinates $(z,w) = (z,w^1,\ldots,w^8)$ as the graphs of the
following {\em shared} six defining equations:
\begin{equation}
\label{models} \aligned w^1-\overline{w}^1&=2\,i\,z\overline{z}, \ \
\ \ \ \ \ \ \ \ \ \ \ \ \ \ \ \ \ \ \ \ \
w^2-\overline{w}^2=2\,i\,(z^2\overline{z}+z\overline{z}^2),
\\
w^3-\overline{w}^3&=2\,(z^2\overline{z}-z\overline{z}^2), \ \ \ \ \ \
 \ \ \
w^4-\overline{w}^4=2\,i\,(z^3\overline{z}+z\overline{z}^3),
\\
w^5-\overline{w}^5&=2\,(z^3\overline{z}-z\overline{z}^3), \ \ \ \ \ \
\ \ \  w^6-\overline{w}^6=2\,i\,z^2\overline{z}^2,
\endaligned
\end{equation}
together with:
\[ \textrm {First model} \  \mathbb M^1: \ \ \
\left\{ \aligned
w^7-\overline{w}^7&=2\,i\,(z^4\overline{z}+z\overline{z}^4),
\\
w^8-\overline{w}^8&=2\,(z^4\overline{z}-z\overline{z}^4), \ \ \ \ \ \
\  {\rm\bf }
\endaligned\right.
\]
\[ \textrm {Second model} \  \mathbb M^2: \
\left\{ \aligned
w^7-\overline{w}^7&=2\,i\,(z^3\overline{z}^2+z^2\overline{z}^3),
\\
w^8-\overline{w}^8&=2\,(z^3\overline{z}^2-z^2\overline{z}^3), \ \ \ \
\ \ \  {\rm\bf }
\endaligned\right.
\]
\[\textrm {Third model} \  \mathbb M^3:
\left\{ \aligned
w^7-\overline{w}^7&=2\,i\,(z^3\overline{z}^2+z^2\overline{z}^3),
\\
w^8-\overline{w}^8&=2\,i\,(z^4\overline{z}+z\overline{z}^4). \ \ \ \
\ \ \ \ \ \
\endaligned\right.
\]
For this, we use the {\sc Maple} package {\tt DifferentialAlgebra} to
carry out the necessary computations. Nevertheless, these
computations involve complex integers and functions and hence it is
not possible to perform directly this package in this respect. That
is why, we need also to utilize a new reduction called {\sl bar-reduction}
that uses complex conjugation to obtain a full remainder ({\em see}
Section \ref{Solving PDE}). Executing long computations, we achieve
the following results. Of course, the main interest of
our algorithmic partly automatize approach is
to open the door to a wealth of other examples
taking inspiration from Beloshapka's universal models.

\begin{Theorem}
The Lie algebras of infinitesimal CR-automorphisms
$\frak{aut}_{CR}(\mathbb M^1)$, $\frak{aut}_{CR}(\mathbb M^2)$ and
$\frak{aut}_{CR}(\mathbb M^1)$ of the three real analytic generic
CR-generic submanifolds $\mathbb M^1, \mathbb M^2$ and $\mathbb M^3$ of
$\mathbb C^{1+8}$ are of dimensions $12,12$ and $11$, respectively
and are generated by the $\mathbb R$-linearly independent real parts
of the following collections of holomorphic vector fields:
\def\theequation{$\mathbb{M}^1$}
\begin{equation}
\footnotesize
\left\{ \aligned
 {\sf X}_i&:=\partial_{w_i}, \ \ \ \ \ {
i=1,\ldots,8},
\\
 {\sf X}_9&:=z\partial_{z}+2\,w^1\partial_{w^1}+3\,w^2\partial_{w^2}+3
\,w^3\partial_{w^3}+4\,w^4\partial_{w^4}+
 \\
&+
 4\,w^5\partial_{w^5}+4\,w^6\partial_{w^6}+5\,w^7\partial_{w^7}+5\,w^8
\partial_{w^8},
 \\
 {\sf
 X}_{10}&:=i\,z\partial_{z}-w^3\partial_{w^2}+w^2\partial_{w^3}-2\,w^5
\partial_{w^4}+2\,w^4\partial_{w^5}-3\,w^8\partial_{w^7}+3\,w^7
\partial_{w^8},
 \\
 {\sf
 X}_{11}&:=\partial_z+2\,i\,z\partial_{w^1}+(4\,w^1+2\,i
\,z^2)\partial_{w^2}+2\,z^2\partial_{w^3}+(3\,w^2+2\,i
\,z^3)\partial_{w^4}+
 \\
 &+
 (3\,w^3+2\,z^3)\partial_{w^5}+2\,w^2\partial_{w^6}+(4\,w^4+2\,i
\,z^4)\partial_{w^7}+(4\,w^5+2\,z^4)\partial_{w^8},
 \\
 {\sf
 X}_{12}&:=i\,\partial_{z}+2\,z\partial_{w^1}+2\,z^2\partial_{w^2}+(4
\,w^1+-2\,i\,z^2)\partial_{w^3}+(-3\,w^3+2\,z^3)\partial_{w^4}+
 \\
 &+
 (3\,w^2-2\,i\,z^3)\partial_{w^5}+2\,w^3\partial_{w^6}+(-4\,w^5+2
\,z^4)\partial_{w^7}+(4\,w^4-2\,i\,z^4)\partial_{w^8};
 \endaligned
 \right.
\end{equation}
\def\theequation{$\mathbb{M}^2$}
\begin{equation}
\footnotesize \left\{ \aligned
 {\sf X}_i&:=\partial_{w_i}, \ \ \ \ \ {
i=1,\ldots,8},
\\
 {\sf X}_9&:=z\partial_{z}+2\,w^1\partial_{w^1}+3\,w^2\partial_{w^2}+3
\,w^3\partial_{w^3}+4\,w^4\partial_{w^4}+
 \\
&+
 4\,w^5\partial_{w^5}+4\,w^6\partial_{w^6}+5\,w^7\partial_{w^7}+5\,w^8
\partial_{w^8},
 \\
 {\sf
 X}_{10}&:=i\,z\partial_{z}-w^3\partial_{w^2}+w^2\partial_{w^3}-2\,w^5
\partial_{w^4}+2\,w^4\partial_{w^5}-
 w^8\partial_{w^7}+w^7\partial_{w^8},
 \\
 {\sf
 X}_{11}&:=\partial_z+2\,i\,z\partial_{w^1}+(4\,w^1+2\,i
\,z^2)\partial_{w^2}+2\,z^2\partial_{w^3}+(3\,w^2+2\,i
\,z^3)\partial_{w^4}+
 \\
 &+
 (3\,w^3+2\,z^3)\partial_{w^5}+2\,w^2\partial_{w^6}+(2\,w^4+6
\,w^6)\partial_{w^7}+2\,w^5\partial_{w^8},
 \\
 {\sf
 X}_{12}&:=i\,\partial_{z}+2\,z\partial_{w^1}+2\,z^2\partial_{w^2}+(4
\,w^1-2\,i\,z^2)\partial_{w^3}+(-3\,w^3+2\,z^3)\partial_{w^4}+
 \\
 &+
 (3\,w^2-2\,i\,z^3)\partial_{w^5}+2\,w^3\partial_{w^6}+2\,w^5
\partial_{w^7}+(-2\,w^4+6\,w^6)\partial_{w^8};
 \endaligned
 \right.
\end{equation}
\def\theequation{$\mathbb{M}^3$}
\begin{equation}
\footnotesize \left\{ \aligned
 {\sf X}_i&:=\partial_{w_i}, \ \ \ \ \ {
i=1,\ldots,8},
\\
 {\sf X}_9&:=z\partial_{z}+2\,w^1\partial_{w^1}+3\,w^2\partial_{w^2}+3
\,w^3\partial_{w^3}+4\,w^4\partial_{w^4}+
 \\
&+
 4\,w^5\partial_{w^5}+4\,w^6\partial_{w^6}+5\,w^7\partial_{w^7}+5\,w^8
\partial_{w^8},
 \\
 {\sf
 X}_{10}&:=\partial_z+2\,i\,z\partial_{w^1}+(4\,w^1+2\,i
\,z^2)\partial_{w^2}+2\,z^2\partial_{w^3}+(3\,w^2+2\,i
\,z^3)\partial_{w^4}+
 \\
 &+
 (3\,w^3+2\,z^3)\partial_{w^5}+2\,w^2\partial_{w^6}+(2\,w^4+6
\,w^6)\partial_{w^7}+(4\,w^4+2\,i\,z^4)\partial_{w^8},
 \\
 {\sf
 X}_{11}&:=i\,\partial_{z}+2\,z\partial_{w^1}+2\,z^2\partial_{w^2}+(4
\,w^1-2\,i\,z^2)\partial_{w^3}+(-3\,w^3+2\,z^3)\partial_{w^4}+
 \\
 &+
 (3\,w^2-2\,i\,z^3)\partial_{w^5}+2\,w^3\partial_{w^6}+2\,w^5
\partial_{w^7}+(-4\,w^5+2\,z^4)\partial_{w^8}.
 \endaligned
 \right.
\end{equation}
\end{Theorem}

The paper is organized as follows. Section \ref{Differentialalgebra}
contains an overview of the necessary background concerning the
theory of differential algebras. In Section \ref{Solving PDE}, we
present the extended Ritt's reduction algorithm and we show how to
utilize the differential algebraic tools to resolve a system of
partial differential equations. In Section \ref{General description},
we present a general method/strategy to compute the Lie algebra
of infinitesimal CR-automorphisms of {\em arbitrary} generic
real analytic CR-generic submanifolds $M \subset
\mathbb C^{n+k}$. Section
\ref{Computations} is devoted to compute in detail the Lie algebra of
infinitesimal CR-automorphisms $\frak{aut}_{CR}(\mathbb M^1)$,
associated to the first model $\mathbb M^1$. However, we do not
present the intermediate computations of $\frak{aut}_{CR}(\mathbb
M^2)$ and $\frak{aut}_{CR}(\mathbb M^3)$ since they are similar to
those of $\mathbb M^1$ and offer no new aspect.
Tables of Lie brackets appear at the end.


\section{Differential algebra preliminaries}
\label{Differentialalgebra}

In this section, we present a brief overview of basic definitions,
notation and results in differential algebra. Two extensive
surveys of this subject are:~\cite{Kolchin,Ritt}.

\begin{Definition}
An operator $\delta : R \rightarrow R$ over the algebraic ring $R$ is
called a {\sl derivation operator}, if for each $a,b \in R$ we have:
\begin{eqnarray*}
\delta(a+b)&=&\delta(a)+\delta(b) \ \ \ \ \ \  \ \ \ \ \textrm{and} \\
\delta(ab)&=&\delta(a)b + a\delta(b).\\
\end{eqnarray*}
A {\sl differential ring} is a pair $(R,\Delta)$ where $R$ is a ring
equipped with a collection $\Delta=\{\delta_1,\ldots,\delta_m\}$ of
commuting derivations operators over it, satisfying:
\[
\delta_i\delta_j a = \delta_j\delta_i a, \ \ \ \ {\scriptstyle
(i,j\,=\,1\ldots m\,; \ \ a\in R).}
\]
 For simplicity, we suppress the dependence on $\Delta$ in the
 notation and denote a differential ring just by $R$. If $m=1$, then
$R$ is called an {\sl ordinary} differential ring; otherwise it will
be called {\sl partially}. An algebraic ideal $I$ of $R$ is called a
{\sl differential ideal} when it is closed under the action of
derivations of $R$, namely $\delta a\in I$ for each $\delta\in\Delta$
and $a\in I$.
\end{Definition}

\begin{Example}
The ring of polynomials $\mathbb C[x_1,\ldots,x_m]$ over the
variables $x_1,\ldots,x_m$ with rational coefficients together with
the set of operators $\partial/\partial x_1,\ldots,\partial/\partial
x_m$ is a differential ring.
\end{Example}

Let $R$ be a differential ring with
$\Delta=\{\delta_1,\ldots,\delta_m\}$. Here, we introduce a
collection of notations in differential algebra  via the
following itemized list:

\begin{itemize}
\smallskip\item[$\bullet$] We denote by $\Theta$ the free multiplicative
commutative semigroup generated by the elements of $\Delta$, namely
\[
\Theta:=\big\{\delta_1^{t_1}\delta_2^{t_2}\ldots\delta_m^{t_m} \ \ \
\
{\scriptstyle t_1,\ldots,t_m\in\mathbb N}\big\}
\]
Each element $\theta=\delta_1^{\alpha_1}\cdots \delta_m^{\alpha_m}$
of $\Theta$ is called a {\sl derivation operator} of $R$ and
furthermore the sum ${\rm ord}(\theta):=\sum_{i=1}^m t_i$ is called
the {\sl order} of $\theta$. Then $\theta a$ is said to be a
derivative of $a\in R$ of order $\ord(\theta)$.

\smallskip\item[$\bullet$] For
an arbitrary subset $S$ of $R$, set $\Theta S:= \{\theta s \mid s \in
S, \theta \in \Theta \}$. It is the smallest subset of $R$ containing
$S$ which is stable under derivation.

\smallskip\item[$\bullet$] An algebraic ideal of $R$ is called a {\em differential
ideal}, if it is closed under the derivation operators. We denote by
$(S)$ and $[S]$ respectively, the smallest algebraic and differential
ideals of $R$ containing $S$. In fact, $[S]=(\Theta S)$. This fact
provides
an algebraic approach to differential ideals which enables one to
employ algebraic means.

\smallskip\item[$\bullet$] For
a field of characteristic zero $K$, a {\sl differential
polynomial ring}:
\[
R
:=
K\{u_1,\ldots,u_n\}
:=
K\big[\Theta U\big]
\]
is the usual commutative
polynomial ring generated by $\Theta U$ over $K$, where
$U:=\{u_1,\ldots,u_n\}$ is the set of {\sl differential
indeterminate}.

\smallskip\item[$\bullet$] For two certain derivatives $\theta
u$ and $\phi u$ of a same differential indeterminate $u$, we denote
by $\lcd(\theta u, \phi u)$ the {\em least common derivative} between
$\theta u$ and $\phi u$, easily seen to be:
\[
\lcd(\theta u, \phi u)=\lcm(\theta, \phi) u.
\]

\end{itemize}\smallskip

In this paper we let
$K$ be a differential field of characteristic zero.

\begin{Definition}
Let $R=K\{U\}$ be a differential polynomial ring with the set of
indeterminates $U=\{u_1,\ldots,u_n\}$. A {\sl ranking} $>$ is an
ordering over $\Theta U$ compatible with the derivation act over
$\Theta U$, in the sense that
for each derivation $\delta \in \Theta$ and for
each $v,w \in \Theta U$ we have:
\begin{itemize}
\item [$\bullet$]{$\delta v > v$,}
\item[$\bullet$]{$v > w \  \Rightarrow \ \delta v > \delta w$.}
\end{itemize}
For each $\theta,\phi \in \Theta$ and $v,w \in U$, a ranking
$>$ for which the statement  $\ord (\theta) > \ord (\phi)$ implies
that $\theta v > \phi w$ is called {\sl orderly}. Simultaneously, if
the assumption $v > w$ gives $\theta v
> \phi w$, then $>$ is called {\sl elimination}. Moreover, for a fixed
ranking $>$
over $\Theta U$ and for a differential polynomial $p\in
R=K\{u_1,\ldots,u_n\}$, the {\sl leader} $\ld(p)$ of $p$ is the
highest derivative appearing in $p$  with respect to $>$. If
$\ld(p)=u$ and $d$ is the degree of $u$ in the expression of $p$
then, the {\sl initial} $\I(p) \in K$ is defined to be the coefficient of
$u^d$ in $p$. Finally, $u^d$ is called the {\sl rank} of $p$, denoted
by $\rank (p)$.
\end{Definition}


\section{Differential algebra and {\sc pde} systems}
\label{Solving PDE}

Each differential polynomial ring $R=K\{U\}$ can be considered as the
conventional polynomial ring $K[\Theta U]$ whose indeterminates are
derivations of $R$. This enables one to use the conventional
algebraic tools and get useful information about the differential
polynomial ring and its differential ideals. In this section, we
employ the {\em Rosenfeld-Gr\"obner} algorithm to discuss a system of
partial differential equations, using algebraic operations. For a
{\sc pde} system $\Sigma\subset R$, the Rosenfeld-Gr\"obner algorithm
presents the radical differential ideal generated by $\Sigma$ as an
intersection of a finite number of differential ideals which are
called {\sl regular differential ideals}. Those are some differential
ideals $I$ represented by a {\sl canonical representative} $C$, {\em
i.e.} a set of differential polynomials which depends only on $I$ and
the given ranking. A canonical representative of the differential
ideal $I$ helps to solve ideal membership problem, which is a key
computational tool to analyze a PDE system.
An implementation of the Rosenfeld-Gr\"obner
algorithm is available in the {\sc Maple} package {\tt diffalg}, and
it was recently renovated into the package {\tt DifferentialAlgebra}.

One of the main contributions of this paper is the use of
Rosenfeld-Gr\"obner algorithm\,\,---\,\,followed by performing some
further algebraic manipulations\,\,---\,\,for considering our {\sc
pde} systems. It is worth emphasizing that the {\sc pde} systems that we
deal with in this paper are {\em linear and admit complex
equations}. Then, we have to equip the Rosenfeld-Gr\"obner algorithm
with a certain operator which enables to treat with such systems.
Moreover, in the considerably significant class of {\em rigid}
CR-generic submanifolds ({\em see} the end of section \ref{General
description} for definition), the under consideration {\sc pde}
systems are not only complex and linear but also {\em with constant
coefficients} and it is therefore reasonable to consider such
systems more seriously.  For brevity, let us call this type of systems
by ${\sf LinCons}$ systems and also similarly, let us call each linear
differential polynomial with constant coefficients ${\sf LinCons}$
polynomial\footnote{Not every linear differential
polynomial is a ${\sf LinCons}$ one, in general. For example, as an
element of $\mathbb C(x,y)[u,v]$, the polynomial
$p:=x^2y\,u_x+2y\,v_{xy}$ is linear while its coefficients are not
constant.}. Computation with the Rosenfeld-Gr\"obner algorithm are
comparatively less expensive.  Furthermore, in this case
{\em there is no need longer to decompose the differential ideal
generated by the system into regular differential ideals} (a
complicated aspect of the general algorithm).

\subsection{Extended Ritt's reduction algorithm}
Let us therefore recall the Ritt's reduction algorithm,
restricted to ${\sf LinCons}$
differential polynomials. Therefore, the reduction algorithm
described here shall be a {\em weak version of the Ritt's reduction
algorithm} in comparison with the version of~\cite{Ritt,
Kolchin}. Let us recall the definition of partial divisibility for
differential polynomials.

\begin{Definition}
Consider two ${\sf LinCons}$ differential polynomials $p_1$ and $p_2$.
We say that {\sl $p_2$ reduces $p_1$ due to the Ritt's reduction
algorithm}, whenever there exists a certain derivation $\theta$ with
$\rank (p_1) = \rank (\theta p_2)$. In this case, the result of
reduction is:
\[
r
:=
p_1
-
\frac{{\I({p_1})}}{{\I({p_2})}}\,
\theta p_2.
\]
When $\theta$
is proper, we call $r$, the {\em partial remainder} of $p_1$ on
division by $p_2$.
\end{Definition}

One notices that in this  definition, if $\rank (p_1) =\rank (p_2)$
then $\theta$ must be the identity element of $\Theta$ and the Ritt's
reduction coincides with the conventional division algorithm for
multivariate polynomial rings.

\begin{Theorem}{\sc (Ritt's reduction theorem)}
\label{Ritt-Thm} Consider a differential polynomial ring
$R=K\{u_1,\ldots,u_n\}$
over a field $K$ of characteristic zero,
let $\Theta$ be the set of derivation
operators and let `$>$'
be a ranking over $\Theta U$. Furthermore, assume that
$p\in R$ is a ${\sf LinCons}$ differential polynomial and
let $Q$ be a
finite set  of ${\sf LinCons}$ differential
polynomials. Then, there exists $r \in R$,
and for each $q\in Q$, there exists
$\theta_q \in \Theta$ and $c_q\in K$ satisfying the
following conditions:

\begin{itemize}
\smallskip\item
$p=\big(\sum_{q\in Q}\,c_q\,\theta_q\, q\big) + r,$

\smallskip\item{$\rank (r) <
\rank (\theta_qq)$, for each $q$ appearing in the summation,}

\smallskip\item{no term of $r$ is a derivation of $\rank (q)$ for each $q \in Q$.}

\end{itemize}
\end{Theorem}

Here, the differential polynomial $r\in R$ is called the {\sl
remainder} of $p$ on division by $Q$. In order to prove this Theorem,
let us display first the Ritt's reduction algorithm ensuing its
assertion. Then, the correctness and termination of this algorithm
proves Theorem \ref{Ritt-Thm}.

\begin{algorithm}[H]
\caption{RittReduction}
\begin{algorithmic}
\label{Ritt}
\REQUIRE{$p\in R$, $Q\subset R$; a finite set and {$>$}; a ranking}
\ENSURE{$r$; a remainder of $p$ on division by $Q$}
\STATE{$h:=p;$}
\STATE{$r:=0;$}
\WHILE{$h\ne 0$}
    \IF{ there is some $q \in Q$ and $\theta \in \Theta$ with $\rank (h)
= \rank (\theta q)$}
    \STATE{$h:=h-\frac{\I(h)}{\I(q)}\,\theta q$;}
    \ELSE
        \STATE{$r:=r+\I(h)\,\rank (h);$}
        \STATE{$h:=h-\I(h)\,\rank (h);$}
    \ENDIF
\ENDWHILE \STATE{{\bf Return} ($r$)}
\end{algorithmic}
\end{algorithm}

\begin{proof}
The termination of this algorithm follows from the well-ordering
property of $>$ ({\em cf.} \cite{Kolchin}). Namely as one observes,
the rank of $h$ decreases as one goes along the steps of the algorithm
and hence, it terminates after a finitely many steps, when we will
have $h=0$.

Now let us consider the correctness of the algorithm. For
this, we claim that the equality:
\begin{eqnarray}
\label{eq:1} p=\big(\sum_{q\in Q}\,c_q\,\theta_q\, q\big) + h + r
\end{eqnarray}
holds at each step of the algorithm. If we prove this claim then, the
final value $h=0$ of $h$ can concludes the assertion. To prove the
claim, we consider two cases:
\begin{description}
\item[Case 1] If a division occurs  by a polynomial, say $q_i$,
then the right hand side of (\ref{eq:1}) is equal to:
\[\Big(\sum_{q\in
Q}\,c_q\,\theta_q\, q + c\, \theta_{q_i}\, q_i\Big) +
\big(h-c \,\theta_{q_i}
q_i)\big) + r
\]
with $c=\I({q_i})/ \I(h)$. Therefore, it visibly is still equal to $p$.
\item[Case 2] If no
division arises, then the right hand side of (\ref{eq:1}) is equal
to:
\[
\sum_{q\in Q}\,c_q\,\theta_q \,q
+
\big(h-\I(h)\, \rank(h)\big) +
\big(r+\I(h)\, \rank(h)\big),
\]
 and this is also equal to $p$.
\end{description}
In both cases, the equation (\ref{eq:1}) is satisfied at each step
of the algorithm. Moreover, one convinces oneself that the last two
assertions of the Theorem hold according to the structure of this
algorithm.
\end{proof}

One should notice that, as long as the field $K$ is the one of real
numbers $\mathbb R$, then the above version of the Ritt's reduction
algorithm works as well to compute the remainder of the division of a
${\sf LinCons}$ differential polynomial by a finite system of ${\sf
LinCons}$ {\sc pde}'s. However, the coefficients of the {\sc pde}
systems that we consider in this paper belong to the field
$\mathbb{C}=\mathbb{R}(i)$ with $i=\sqrt{-1}$. In this case, we need to
perform also the {\em complex conjugation} to obtain a full remainder.
For this, we need the following definition and theorem.

\begin{Definition}
Let $R:=\mathbb{C}\{u_1,\ldots,u_n\}$ be a differential polynomial
ring with $u_j = {\rm Re}(u_j) + i\, {\rm Im}(u_j)$ for
$j=1,\ldots,n$ as the unknown functions. We define the {\sl bar
operation} $\overline{\bullet} : R \rightarrow R$ by:
\begin{itemize}
\item{$\overline{{\rm Re}(a) + i\,{\rm Im}(a)} = {\rm Re}(a) -
i\,{\rm Im}(a)$ for each $a\in \mathbb{C}$,} \item{$\overline{{\rm
Re}(u_j) + i\,{\rm Im}(u_j)} = {\rm Re}(u_j) - i\,{\rm Im}(u_j)$ for
each $j=1,\ldots,n$.}
\end{itemize}
\end{Definition}

As a matter of fact, the bar operator is compatible with the
derivations; namely, for each (real) $\theta \in \Theta$ and each
$j=1,\ldots,n$, one has:
\[
\overline{\theta u_j} = \theta \overline{u_j}.
\]
This allows one to insert the bar reduction operator in the Ritt's
reduction algorithm.

\begin{Theorem}
\label{ModRitt-Thm}
{\sc (Extended Ritt's reduction theorem)} Consider a differential
polynomial ring $R=\mathbb{C}\{u_1,\ldots,u_n\}$ over the field of
complex numbers, let $\Theta$ be the set of its derivation operators,
and let `$>$' be a ranking over $\Theta U$. Furthermore, assume that
$p\in R$ is a ${\sf LinCons}$ differential polynomial and let $Q$ be a
finite set of ${\sf LinCons}$ differential polynomials. Then, there
exists $r \in R$, and for each $q\in Q$, there exists $\theta_q,
\theta_{\overline{ q}} \in
\Theta$ and $c_q, c_{\overline{ q}}\in K$
satisfying the following conditions:

\begin{itemize}

\smallskip\item
$p= \big( \sum_{q\in Q}\,c_q\,\theta_q\, q \big) +
\big( \sum_{q\in
Q}\,c_{\overline{q}}\,\theta_{\overline{q}} \,\overline{q} \big) + r,$

\smallskip\item
for each $q$ appearing in this summation we have:
\[\rank(r)
<
{\rm min}_>
\big\{\rank(\theta_qq),
\rank(\theta_{\overline{q}}{\overline{q}})
\big\},
\]

\smallskip\item{no term of $r$ is a derivation of either $\rank(q)$ or
$\rank(\overline{q})$ for each $q \in Q$.}
\end{itemize}
\end{Theorem}

Here, the differential polynomial $r$ is called the {\sl
full-remainder} of $p$ on division by $Q$. The proof of this Theorem
is similar to that of Theorem \ref{Ritt-Thm}. Let us display the
following algorithm like  Algorithm \ref{Ritt} extended by the bar
reduction.

\begin{algorithm}[H]
\caption{ExtendedRittReduction}
\begin{algorithmic}
\label{Mod-Ritt} \REQUIRE{$p\in R$, $Q\subset R$; a finite set and
{$>$}; a ranking}
\ENSURE{$r$; a remainder of $p$ on division by $Q$}
\STATE{$h:=p;$}
\STATE{$r:=0;$}
\WHILE{$h\ne 0$}
    \IF{ there is some $q \in Q$ and $\theta \in \Theta$ with $\rank(h)
= \rank(\theta q)$}
        \STATE{$h:=h-\frac{\I(h)}{\I(q)}\,\theta q$;}
    \ELSE
        \IF{ there is some $q \in Q$ and $\theta \in \Theta$ with
$\rank(h) = \rank(\theta \overline{q})$}
            \STATE{$h:=h-\frac{\I(h)}{\I(q)}\,\theta \overline{q}$;}
        \ELSE
            \STATE{$r:=r+\I(h) \,\rank(h);$}
            \STATE{$h:=h-\I(h)\, \rank(h);$}
        \ENDIF
    \ENDIF
\ENDWHILE \STATE{{\bf Return} ($r$)}
\end{algorithmic}
\end{algorithm}

\subsection{Rosenfeld-Gr\"obner algorithm} As
mentioned before, to solve a PDE system $\Sigma$,
we use the Rosenfeld-Gr\"obner algorithm to decompose the radical of
$[\Sigma]$, into some new PDE systems, presented by
explicit generators.
These generators have novel properties which leads to do a complete
analysis of $\Sigma$.
The main Rosenfeld-Gr\"obner algorithm as presented in
\cite{Blop95,Boulier} requires some recursive loops to construct the
mentioned decomposition. However, as
we consider the Rosenfeld-Gr\"obner algorithm in the
special case of the ${\sf LinCons}$ {\sc pde} systems,
{\em there is no need longer to
decompose the differential ideal generated by the system into regular
differential ideals}. So, it is convenient to provide an adapted and
computationally simpler version of the Rosenfeld-Gr\"obner algorithm
to deal with just ${\sf LinCons}$ {\sc pde}s. Let us call this
algorithm by {\sl LRG} which stands for the {\sl \bf L}{\sf incons} {\sl
\bf R}osenfeld-{\sl \bf G}r\"obner algorithm. First, we need the
definition
of $\Delta$-polynomial\,\,---\,\,similar in spirit to that in
Gr\"obner bases theory\,\,---\,\,which plays a crucial role.

\begin{Definition}
Consider two ${\sf LinCons}$ differential polynomials $p_1$ and $p_2$
with $\ld(p_i)=\theta_i\, u_i$, $i=1,2$. Then, the $\Delta$-{\sl
polynomial} of $p_1$ and $p_2$ is defined as:
\begin{eqnarray*}
\Delta(p_1,p_2) =\left\{
\begin{array}{ll}
 \lc(p_2)\frac{\theta_{1,2}}{\theta_1} p_1 -
\lc(p_1)\frac{\theta_{1,2}}{\theta_2} p_2 & u_1=u_2,\\
0& u_1 \ne u_2,\\
\end{array}
\right.
\end{eqnarray*}
where $ \theta_{1,2} = \lcd(\theta_1,\theta_2)$.
\end{Definition}

 The aim of calculating the $\Delta$-polynomial of two differential
polynomials is
 in fact to remove their leading derivatives to obtain (probably)
  a new leading derivative.

  Now, let us describe the LRG algorithm.
  If $\Sigma \subset R$ is a subset of a differential ring $R$ (in fact,
a {\sc pde} system) then,
   $[\Sigma]$ denotes the smallest
differential ideal of $R$, containing $\Sigma$ and {\em closed under
complex conjugation}.

\begin{algorithm}[H]
\caption{LRG algorithm}
\begin{algorithmic}
\label{LRG} \REQUIRE{$\Sigma$; a finite set of ${\sf LinCons}$
differential polynomials, {$>$}; a ranking} \ENSURE{$G$; a canonical
representative for $[\Sigma]$} \STATE{$G:=\Sigma$;}
\STATE{$P:=\{\{p_1,p_2\} \mid p_1,p_2 \in G\}$;} \WHILE{$P\ne \{\}$}
    \STATE{Select and remove $\{p_1,p_2\} \in P$;}
    \STATE{$h:=\Delta(p_1,p_2)$;}
    \STATE{$r:=$ExtendedRittReduction$(h,G,>)$;}
    \IF{$r\ne 0$ }
        \STATE{$P:=P \cup \{\{r,g\} \mid g\in G\}$;}
        \STATE{$G:=G \cup \{r\}$;}
    \ENDIF
\ENDWHILE
\STATE{{\bf Return} ($G$)}
\end{algorithmic}
\end{algorithm}

The following theorem shows the termination and correctness of the
algorithm.

\begin{Theorem}
The following statements hold:

\begin{itemize}

\smallskip\item[(a)]{
LRG algorithm terminates in a finite number of steps.}

\smallskip\item[(b)]{
If $G$ is the canonical representative of $[\Sigma]$ then,
any full-remainder of a ${\sf LinCons}$ differential polynomial $p\in
R$ on division by $G$ is zero if and only if $p \in [\Sigma]$.}
\end{itemize}

\end{Theorem}

\begin{proof}
(a) The termination of the algorithm is guaranteed by the
Ritt-Raudenbush basis Theorem (the analogue of the Hilbert basis
theorem for polynomial rings). According to this theorem, since
$K=\mathbb{C}$ is Noetherian with respect to the radical
differential ideals, then $R$ is too ({\em cf.} \cite{Kolchin}).

Now,
using {\em reductio ad absurdum}, let us assume that the algorithm does not
terminate for a finite set $\Sigma$. Thus, we have an ascending chain
of ideals $[\Sigma_1] \subset [\Sigma_2] \subset \cdots$ which does
not stabilize where $\Sigma_i$ is the set of leading derivatives of
the differential polynomials at the $i$-th step of the execution of
the algorithm (by the $i$-th step we mean computing the $i$-th new
polynomial and adding it to $G$). One can observe that $\Sigma_i$
contains ${\sf LinCons}$ polynomials and therefore $[\Sigma_i]$ is a
radical ideal, namely a radical differential ideal\footnote{One
should notice that since the base field contains $\mathbb{Q}$, then
the radical of a differential ideal is a radical differential ideal
({\em cf.} {\cite{Sofi}, Proposition 2.1}).}. This contradicts the
Noetherianity of $R$ and so, proves the termination.

(b) To prove the correctness of the algorithm, it is sufficient to
show that if $p \in [\Sigma]$ is a ${\sf LinCons}$ differential
polynomial then its full-remainder  on division by $G$ is zero.
Since $p \in [\Sigma]$, then one can conclude by
Rosenfeld's Lemma ({\em see} \cite{Kolchin}, Chap. III, Sect. 8,
Lemma 5), that a partial remainder of $p$, say $p'$, on division by $G$
belongs to $(\Sigma \cup \overline{\Sigma})$. On the other hand, the set
of $\Delta$-polynomials
contains also S-polynomials ({\em cf.} \cite{becker})\,\,---\,\,note
that the polynomials, under consideration in this paper, are ${\sf
LinCons}$. Moreover, in {\sc ExtendedRittReduction} algorithm, we
consider the complex conjugation of any computed polynomial. These
imply that
\[
\text{\sc ExtendedRittReduction}(p',G,>)=0,
\]
 according to the Buchberger's criterion ({\em see} \cite{becker},
Theorem 5.48). Therefore, any full-remainder of $p$ will be equal to
zero.
\end{proof}

At the end of this section, let us illustrate with the help of an
example, how one can employ the {\sc Maple} package  {\tt
DifferentialAlgebra}\footnote{For more details on this package, we
refer to {\tt http://www.maplesoft.com/support/help}} to handle and
solve a ${\sf LinCons}$ {\sc pde} system.

\begin{Example}
 Consider the following ${\sf LinCons}$ {\sc pde} system $\Sigma\subset
\mathbb Q(x,y)[u,v]$:
\begin{eqnarray*}
\Sigma:= \left\{
\begin{array}{rcl}
 \frac{\partial^2}{\partial y^2} u(x,y) - \frac{\partial^2}{\partial
x^2} v(x,y) - \frac{\partial^2}{\partial x \partial y} v(x,y)&=&0,\\
 \frac{\partial^2}{\partial x^2} v(x,y) - \frac{\partial^2}{\partial
y^2} v(x,y) + v(x,y) &=&0,\\
 \frac{\partial^2}{\partial x^2} u(x,y) - \frac{\partial^2}{\partial x
\partial y} u(x,y) &=&0.\\
\end{array}
\right.
\end{eqnarray*}
where $\mathbb{Q}(x,y)$ is the field of fractions of the polynomial
ring $\mathbb{Q}[x,y]$. First, we must call the desired package to
make it available:
\begin{verbatim}
[> with(DifferentialAlgebra);
\end{verbatim}
and continue with the following three steps.

\begin{itemize}
\item[Step 1.] In this step, we set:
\[
\Sigma:=\{p_1:=u_{y,y}-v_{x,x}-v_{x,y},\,\,\,
p_2:=v_{x,x}-v_{y,y}+v,\,\,\,p_3:=u_{x,x}-u_{x,y}\}.
\]
 Furthermore, we define the differential ring $\mathbb{Q}(x,y)[u,v]$
equipped with a
ranking as follows.
\begin{verbatim}
[> R:=DifferentialRing(blocks=[[u, v]], derivations=[x, y]);
\end{verbatim}
This ring has $u,v$ as unknown functions of the variables $x,y$ and
also, it admits the set of derivations $\Delta=\{\frac{\partial}{\partial
x},\frac{\partial}{\partial y}\}$. Moreover, here the picked ranking
is orderly, namely  the following statement holds:
\[
\theta u < \delta v \Leftrightarrow \ord \ \theta < \ord \ \delta.
\]
for two arbitrary elements $\delta, \theta$ of
$\Theta=\{\frac{\partial^{i+j}}{\partial x^i \partial y^j}, \ \
i,j\in\mathbb N\}$.
 \item[Step 2.] Next, we employ the
Rosenfeld-Gr\"obner algorithm to compute a canonical representative
for $[\Sigma]$,:
\begin{verbatim}
[> CP := RosenfeldGroebner([p1, p2, p3], R);
\end{verbatim}
Now to see the equations of, for example, the first component of $CP$,
we enter the following command:
\begin{verbatim}
[> Equations(CP[1]);
\end{verbatim}
The result of this line is:
\[
\{u_{x, x}-u_{x, y}, u_{y, y}-v_{x, y}, v_{x, x}, v_{y, y}-v\}.
\]
\item[Step 3.] Finally, we find the formal power series solutions
corresponding to the functions $u,v$ around the point $(0,0)$, up to
the order $3$:
\begin{verbatim}
PowerSeriesSolution(CP, order = 3);
\end{verbatim}
and as the results it returns:
\begin{eqnarray*}
u(x, y) &=& u(0, 0)+u_y(0, 0)y+u_x(0, 0)x+1/2v_{x,y}(0, 0)y^2+\\
&& +u_{x,y}(0, 0)xy+1/2u_{x,y}(0, 0)x^2+1/6u_x(0, 0)y^3, \\
v(x, y) &=& v(0, 0)+v_y(0, 0)y+v_x(0, 0)x+1/2v(0, 0)y^2+\\
&&+v_{x,y}(0, 0)xy+1/6v_y(0, 0)y^3+1/2v_x(0, 0)xy^2,
\end{eqnarray*}
\end{itemize}
which is in fact the solution set of the above {\sc pde} system
$\Sigma$.
\end{Example}

\begin{Remark}
It is worth noting that, if a {\sc pde} system contains some complex
conjugates of unknown functions, then we can use only the {\sc Maple}
package  {\tt DifferentialAlgebra} and without implementing {\sc LRG}
algorithm to analyze it. For this purpose, we associate to each
unknown function a {\em tag variable} as its complex conjugate, and
we add it in the definition of the base differential ring.
Furthermore, we add the complex conjugate of each input differential
polynomial to $\Sigma$. Then, the result of this approach is the same
as the output of {\sc LRG} algorithm. However, the complexity of this
computation is growing (and higher than {\sc LRG} algorithm) when we add
a new tag variable.
\end{Remark}

\section{Infinitesimal Lie algebra of
\\
real analytic CR-generic submanifolds}
\label{General description}

\subsection{Effective tangency equations}
Hereafter, by $M\subset\mathbb C^{n+k}$ we mean a
real analytic CR-generic submanifold of
CR-dimension $n$ and of codimension $k$; recall
(\cite{Baouendi, MerkerPorten}) that
{\sl CR-genericity}
means that $TM + JTM = T\C^{ n+k} \big\vert_M$, where
$J$ is the standard complex structure (multiplication by $i$).
In order to
compute $\mathfrak{aut}_{CR} (M)$ for an explicitly given such
submanifold $M \subset \mathbb C^{n+k}$, it is most convenient to
work with {\sl complex defining equations} of the specific shape
(\cite{ MerkerPorten, AMS, 5-cubic}):
\[
\overline{w}_j+w_j = \overline{\Xi}_j(\overline{z},z,w) \ \ \ \ \ \ \
\ \ \ \ \ \ {\scriptstyle{(j\,=\,1\,\cdots\,k)},}
\]
where the coordinates $(z, w) = (z_1, \dots, z_n, w_1, \dots, w_k)$
are centered at the origin $0 \in M$ and where
$T_0 M = \{ \overline{ w}_j + w_j = 0\colon\, j = 1, \dots, k\}$,
so that all $\overline{ \Xi}_j = {\rm O} (2)$.

A general $(1,0)$ vector field having holomorphic coefficients:
\[
{\sf X} = \sum_{ i=1}^n\, Z^i ( z, w)\, \frac{\partial}{ \partial
z_i} + \sum_{ l=1}^k\, W^l( z, w) \, \frac{\partial}{ \partial w_l}
\]
 is called an {\sl infinitesimal
CR-automorphism} of $M$, whenever it is {\em tangent} to it, namely
whenever $(\sf X+\overline{\sf X})|_M\equiv 0$. Concretely and more
precisely, the condition that a holomorphic vector field $\sf X$
belongs to $\mathfrak{aut}_{CR} ( M)$ means that for $j=1,\ldots,k,$
the differentiated equation:
\[
\footnotesize \aligned 0 = (\overline{\sf X}+{\sf X}) & \big[
\overline{w}_j+w_j - \overline{\Xi}_j(\overline{z},z,w) \big] =
\\
& = \overline{\sf X}\,\big[\overline{w}_j+w_j -
\overline{\Xi}_j(\overline{z},z,w) \big] + {\sf X}\,
\big[\overline{w}_j+w_j - \overline{\Xi}_j(\overline{z},z,w) \big]
\\
& = \overline{W}^j(\overline{z},\overline{w}) -
\sum_{i=1}^n\,\overline{Z}^i(\overline{z},\overline{w})\,
\frac{\partial\overline{\Xi}_j}{\partial \overline{z}_i}
(\overline{z},z,w) +
\\
& \ \ \ \ \ + W^j(z,w) - \sum_{i=1}^n\, Z^i(z,w)\,
\frac{\partial\overline{\Xi}_j}{\partial z_i}(\overline{z},z,w) -
\sum_{l=1}^k\,W^l(z,w)\, \frac{\partial\overline{\Xi}_j}{\partial
w_l}(\overline{z},z,w)
\\
\endaligned
\]
should vanish for every $(z, w) \in M$. On the other hand, this
condition holds if and only if, after extrinsic complexification and
replacement of $\overline{ w}$ by $- w + \overline{ \Xi} ( \overline{
z}, z, w)$, the $k$ power series obtained in $\mathbb C \{ \overline{
z}, z, w \}$ vanish identically, and this yields the
{\sl tangency equations}:
\begin{equation}
\label{initial-tangency-general} \footnotesize \aligned 0 & \equiv
\overline{W}^j\big(\overline{z},\,-w
+\overline{\Xi}(\overline{z},z,w)\big)
- \sum_{i=1}^n\,\overline{Z}^i\big(
\overline{z},\,-w+\overline{\Xi}(\overline{z},z,w)\big)\,
\frac{\partial\overline{\Xi}_j}{\partial\overline{z}_i}
(\overline{z},z,w) +
\\
& \ \ \ \ \ + W^j(z,w) - \sum_{i=1}^n\,Z^i(z,w)\,
\frac{\partial\overline{\Xi}_j}{\partial z_i} (\overline{z},z,w) -
\sum_{l=1}^k\, W^l(z,w)\, \frac{\partial\overline{\Xi}_j}{\partial w}
(\overline{z},z,w)
\\
& \ \ \ \ \ \ \ \ \ \ \ \ \ \ \ \ \ \ \ \ \ \ \ \ \ \ \ \ \ \ \ \ \ \
\ \ \ \ \ \ \ \ \ \ {\scriptstyle{(j=1,\ldots,k)}}.
\endaligned
\end{equation}

\begin{Proposition}
\label{first step} The holomorphic vector field:
\[
{\sf X} = \sum_{i=1}^n\,Z^i(z,w) \frac{\partial}{\partial z_i} +
\sum_{l=1}^k\,W^l(z,w) \frac{\partial}{\partial w_l}
\]
is an infinitesimal CR-automorphism of a generic real analytic
CR-generic submanifold $M\subset\mathbb
C^{n+k}$ represented in coordinates
$(z,w )= (z_1,\ldots,z_n,w_1,\ldots,w_k)$ as the
graph of the $k$ complex defining functions:
\[
w_j+\overline{w}_j=\overline{\Xi}_j(\overline{z},z,w)  \ \ \ \ \ \ \
{\scriptstyle{(j=1,\ldots,k)}}
\]
 if and only if its coefficients $Z^i(z,w)$ and $W^l(z,w)$ satisfy the
$k$ equations \thetag{\ref{initial-tangency-general}}.
\end{Proposition}

\subsection{General formulae}

According to Proposition \ref{first step}, to find the infinitesimal
CR-automorphisms $\sf X$ associated to a CR-generic submanifold $M$ of
$\mathbb C^{n+k}$, it is necessary and sufficient to determine the
holomorphic functions $Z^i(z,w)$ and $W^l(z,w)$ satisfying the
tangency equations \thetag{\ref{initial-tangency-general}}. Then, the
main question that immediately arises here is to ask: {\em How can
one specify such functions?} To answer this question, we focus our
attention on providing an effective algorithm. To this aim, first we
introduce the expansions of the coefficients of such a sought ${\sf
X}$ with respect to the powers of $z$:
\begin{equation}
\label{Taylor} \aligned Z^i(z,w) = \sum_{\alpha\in\mathbb
N^n}\,z^\alpha\,Z^{i,\alpha}(w) \ \ \ \ \ \text{\rm and} \ \ \ \ \
W^l(z,w) = \sum_{\alpha\in\mathbb N^n}\,z^\alpha\,W^{l,\alpha}(w),
\endaligned
\end{equation}
where the $Z^{i, \alpha}( w)$ and the $W^{ l, \alpha} ( w)$ are local
holomorphic functions. We will show that the identical vanishing of
the $k$ equations~\thetag{ \ref{initial-tangency-general}} in
$\mathbb C \{ \overline{ z}, z, w\}$ is equivalent to a certain (in
general complicated) linear system of partial differential equations
involving the $\frac{
\partial^{ \vert \gamma \vert} Z^{k, \alpha}}{
\partial w^\gamma}(w)$,
the $\frac{ \partial^{ \vert \gamma' \vert} Z^{ k', \alpha'}}{
\partial w^{ \gamma'}}(w)$,
the $\frac{ \partial^{ \vert \gamma'' \vert} W^{ l, \alpha'''}}{
\partial w^{ \gamma''}}(w)$
and the $\frac{ \partial^{ \vert \gamma''' \vert} W^{ l',
\alpha'''}}{
\partial w^{ \gamma'''}}(w)$.

Inserting these expansions with respect to the powers of $z$ in the
tangency equations \thetag{\ref{initial-tangency-general}}, we get:
\[
\footnotesize \aligned 0 & \equiv \sum_{\alpha\in\mathbb
N^n}\,\overline{z}^\alpha\, \overline{W}^{j,\alpha}
\big(-w+\overline{\Xi}\big) - \sum_{i=1}^n\,\sum_{\alpha\in\mathbb
N^n}\, \overline{z}^\alpha\, \overline{Z}^{i,\alpha}
\big(-w+\overline{\Xi}\big)\,
\frac{\partial\overline{\Xi}_j}{\partial\overline{z}_i}
(\overline{z},z,w) +
\\
& \ \ \ \ \ + \sum_{\beta\in\mathbb N^n}\, z^\beta\,W^{j,\beta}(w) -
\sum_{i=1}^n\,\sum_{\beta\in\mathbb N^n}\, z^\beta\,Z^{i,\beta}(w)\,
\frac{\partial\overline{\Xi}_j}{\partial z_i} (\overline{z},z,w) -
\sum_{l=1}^k\,\sum_{\beta\in\mathbb N^n}\, z^\beta\,W^{l,\beta}(w)\,
\frac{\partial\overline{\Xi}_j}{\partial w_l} (\overline{z},z,w)
\\
& \ \ \ \ \ \ \ \ \ \ \ \ \ \ \ \ \ \ \ \ \ \ \ \ \ \ \ \ \ \ \ \ \ \
\ \ \ \ \ \ \ \ \ \ \ \ \ \ \ \ \ \ \ \ \ \ \ \ \ \
{\scriptstyle{(j\,=\,1\,\cdots\,k)}}.
\endaligned
\]
Since in these equations, $w$ is the argument both of all the $Z^{i,
\beta}$ and of all the $W^{ l, \beta}$ appearing in the second line,
one should arrange that the same argument $w$ takes place inside the
functions $\overline{ W}^{ j, \alpha}$ and $\overline{ Z}^{i,
\alpha}$ appearing in the first line. Thus, one is led, for an
arbitrary converging holomorphic power series $\overline{ A} =
\overline{ A} ( w) = \sum_{ \gamma\in \mathbb N^k}\, \frac{
\partial^{ \vert \gamma \vert} \overline{ A}}{
\partial w^\gamma} ( 0)\, w^\gamma$, to apply
the well known basic infinite Taylor series formulae under the
following slightly artificial form:
\begin{equation}
\label{Taylor-A} \small \aligned
\overline{A}\big(-w+\overline{\Xi}\big) & =
\overline{A}\big(w+(-2w+\overline{\Xi})\big)
\\
& = \sum_{\gamma\in\mathbb N^k}\,
\frac{\partial^{\vert\gamma\vert}\overline{A}}{
\partial w^\gamma}(w)\,
\frac{1}{\gamma!}\,
\big(-2w+\overline{\Xi}(\overline{z},z,w)\big)^\gamma.
\endaligned
\end{equation}
When one does this, one transforms the first lines of the previous
$k$ equations as follows:
\begin{equation}
\label{before-expansion-z-bar-z} \footnotesize \aligned 0 & \equiv
\sum_{\alpha\in\mathbb N^n}\, \sum_{\gamma\in\mathbb N^k}\,
\frac{1}{\gamma!}\, \overline{z}^\alpha\,
\big(-2w+\overline{\Xi}(\overline{z},z,w)\big)^\gamma\,
\frac{\partial^{\vert\gamma\vert}\overline{W}^{j,\alpha}}{
\partial w^\gamma}(w)
-
\\
& \ \ \ \ \ - \sum_{i=1}^n\,\sum_{\alpha\in\mathbb
N^n}\,\sum_{\gamma\in\mathbb N^k}\, \frac{1}{\gamma!}\,
\overline{z}^\alpha\,
\big(-2w+\overline{\Xi}(\overline{z},z,w)\big)^\gamma\,
\frac{\partial^{\vert\gamma\vert}\overline{Z}^{i,\alpha}}{
\partial w^\gamma}(w)
+
\\
& \ \ \ \ \ + \sum_{\beta\in\mathbb N^n}\, z^\beta\,W^{j,\beta}(w) -
\sum_{i=1}^n\,\sum_{\beta\in\mathbb N^n}\, z^\beta\,Z^{i,\beta}(w)\,
\frac{\partial\overline{\Xi}_j}{\partial z_i} (\overline{z},z,w) -
\sum_{l=1}^k\,\sum_{\beta\in\mathbb N^n}\, z^\beta\,W^{l,\beta}(w)\,
\frac{\partial\overline{\Xi}_j}{\partial w_l} (\overline{z},z,w)
\\
& \ \ \ \ \ \ \ \ \ \ \ \ \ \ \ \ \ \ \ \ \ \ \ \ \ \ \ \ \ \ \ \ \ \
\ \ \ \ \ \ \ \ \ \ \ \ \ \ \ \ \ \
{\scriptstyle{(j\,=\,1\,\cdots\,k)}}.
\endaligned
\end{equation}
But still, we must expand and reorganize everything in terms of the
powers $\overline{ z}^\alpha\, z^\beta$ of $(\overline{ z}, z)$. At
first, we must do this for the multipowers:
\[
\big(-2w + \overline{
\Xi} ( \overline{ z}, z, w)\big)^\gamma
=
\prod_{j=1}^k\,
\Big(
-2w_j
+
\overline{\Xi}_j\big(\overline{z},z,w\big)
\Big)^{\gamma_j}.
\]

\subsection{Expansion, reorganization and
associated ${\sf LinCons}$ {\sc pde} system}
To begin with, let us denote the
$(\overline{ z}, z)$-power series expansion of $-2w_j + \overline{
\Xi}_j$ by:
\begin{equation}
\label{Xi-expansion}
 -2w_j + \overline{\Xi}_j(\overline{z},z,w) =
\sum_{\alpha\in\mathbb N^n}\,\sum_{\beta\in\mathbb N^n}\,
\overline{z}^\alpha\,z^\beta\,
\overline{\Xi}_{j,\alpha,\beta}^\sim(w) \ \ \ \ \ \ \ \ \ \ \ \ \
{\scriptstyle{(j\,=\,1\,\cdots\,k)}},
\end{equation}
with the understanding that the coefficients of the expansion of
$\overline{ \Xi}_j$ would be denoted plainly $\overline{ \Xi}_{ j,
\alpha, \beta} ( w)$, without $\sim$ sign.
Reminding $\overline{\Xi}_j = {\rm O}(2)$, we adopt the convention
that in this right-hand side, the $\overline{ \Xi}_{ j,\alpha,
\beta}^\sim ( w)$ for $\alpha = \beta = 0$ comes not from $\overline{
\Xi}_j$ itself, but just from the supplementary first-order term $-2\,
w_j$.

Thus, denoting:
\[
\gamma = (\gamma_1,\gamma_2,\dots,\gamma_k) \in \mathbb N^k,
\]
we may expand explicitly the exponentiated product under
consideration, and the intermediate, detailed computations read as
follows:
\[
\footnotesize \aligned & \prod_{j=1}^k\, \big(-2\,w_j +
\overline{\Xi}_j(\overline{z},z,w)\big)^{\gamma_j} =
\\
& = \prod_{j=1}^k\, \bigg( \sum_{\alpha\in\mathbb
N^n}\,\sum_{\beta\in\mathbb N^n}\, \overline{z}^\alpha\,z^\beta\,
\overline{\Xi}_{j,\alpha,\beta}^\sim(w) \bigg)^{\gamma_j}
\\
& = \prod_{j=1}^k\, \bigg[ \sum_{\alpha\in\mathbb
N^n}\,\sum_{\beta\in\mathbb N^n}\, \overline{z}^\alpha\,z^\beta\,
\bigg( \sum_{\alpha_1+\cdots+\alpha_{\gamma_j}=\alpha \atop
\beta_1+\cdots+\beta_{\gamma_j}=\beta}\,
\overline{\Xi}_{j,\alpha_1,\beta_1}^\sim(w) \cdots\,
\overline{\Xi}_{j,\alpha_{\gamma_j},\beta_{\gamma_j}}^\sim(w) \bigg)
\bigg]
\\
 & = \sum_{\alpha\in\mathbb N^n}\,\sum_{\beta\in\mathbb
N^n}\, \overline{z}^\alpha\,z^\beta \bigg[
\sum_{\alpha^1+\cdots+\alpha^k=\alpha \atop
\beta_1+\cdots+\beta^k=\beta}\,
\sum_{\alpha_1^1+\cdots+\alpha_{\gamma_1}^1=\alpha^1 \atop
\beta_1^1+\cdots+\beta_{\gamma_1}^1=\beta^1}\, \cdots\,
\sum_{\alpha_1^k+\cdots+\alpha_{\gamma_k}^k=\alpha^k \atop
\beta_1^k+\cdots+\beta_{\gamma_k}^k=\beta_k}
\\
& \ \ \ \ \ \ \ \ \ \ \ \ \ \ \ \ \ \ \ \ \ \ \ \ \ \ \ \ \ \ \ \ \ \
\ \overline{\Xi}_{1,\alpha_1^1,\beta_1^1}^\sim(w) \cdots
\overline{\Xi}_{1,\alpha_{\gamma_1}^1,\beta_{\gamma_1}^1}^\sim(w)
\cdots\cdots\, \overline{\Xi}_{k,\alpha_1^k,\beta_1^k}^\sim(w) \cdots
\overline{\Xi}_{k,\alpha_{\gamma_k}^k,\beta_{\gamma_k}^k}^\sim(w)
\bigg]
\\
& =: \sum_{\alpha\in\mathbb N^n}\,\sum_{\beta\in\mathbb N^n}\,
\overline{z}^\alpha\,z^\beta\, \mathcal{A}_{\alpha,\beta,\gamma}
\Big( \big\{
\overline{\Xi}_{\widehat{j},\widehat{\alpha},\widehat{\beta}}^\sim(w)
\big\}_{\widehat{j}\in\mathbb N,\widehat{\alpha}\in\mathbb
N^n,\widehat{\beta}\in\mathbb N^n} \Big),
\endaligned
\]
where we introduce a collection of certain polynomial functions
$\mathcal{ A}_{ \alpha, \beta, \gamma}$ of all the $\overline{ \Xi}_{
\widehat{ j}, \widehat{ \alpha}, \widehat{ \beta}}^\sim ( w)$ that
appear naturally in the large brackets of the penultimate equality,
namely where we set:
\[
\small \aligned \mathcal{A}_{\alpha,\beta,\gamma} \Big( \big\{
\overline{\Xi}_{\widehat{j},\widehat{\alpha},\widehat{\beta}}^\sim(w)
& \big\}_{\widehat{j}\in\mathbb N,\widehat{\alpha}\in\mathbb
N^n,\widehat{\beta}\in\mathbb N^n} \Big) :=
\sum_{\alpha^1+\cdots+\alpha^k=\alpha \atop
\beta_1+\cdots+\beta^k=\beta}\,
\sum_{\alpha_1^1+\cdots+\alpha_{\gamma_1}^1=\alpha^1 \atop
\beta_1^1+\cdots+\beta_{\gamma_1}^1=\beta^1}\, \cdots\,
\sum_{\alpha_1^k+\cdots+\alpha_{\gamma_k}^k=\alpha^k \atop
\beta_1^k+\cdots+\beta_{\gamma_k}^k=\beta_k}
\\
& \overline{\Xi}_{1,\alpha_1^1,\beta_1^1}^\sim(w) \cdots
\overline{\Xi}_{1,\alpha_{\gamma_1}^1,\beta_{\gamma_1}^1}^\sim(w)
\cdots\cdots\, \overline{\Xi}_{k,\alpha_1^k,\beta_1^k}^\sim(w) \cdots
\overline{\Xi}_{k,\alpha_{\gamma_k}^k,\beta_{\gamma_k}^k}^\sim(w).
\endaligned
\]
At present, coming back to the $k$ equations~\thetag{
\ref{before-expansion-z-bar-z}} we left momentarily untouched, we see
that in them, five sums are extant and we now want to expand and to
reorganize properly each one of these sums as a $(\overline{ z},
z)$-power series of the form:
\[
\sum_{ \alpha \in \mathbb N^n} \, \sum_{ \beta \in
\mathbb N^n}\, \overline{ z}^\alpha \, z^\beta \big(
\text{\footnotesize\sf coeff}_{j,\alpha,\beta} \big).
\]
For the sum in~\thetag{ \ref{before-expansion-z-bar-z}}, we therefore
compute, changing in advance the index $\alpha$ to $\alpha'$:
\begin{equation}
\label{I} \footnotesize \aligned & \sum_{\alpha'\in\mathbb
N^n}\,\sum_{\gamma\in\mathbb N^k}\,
\frac{1}{\gamma!}\,\overline{z}^{\alpha'}\,
\frac{\partial^{\vert\gamma\vert}\overline{W}^{j,\alpha'}}{
\partial w^\gamma}(w)
\big(-2w+\overline{\Xi}(\overline{z},z,w)\big)^\gamma\, =
\\
& = \sum_{\alpha'\in\mathbb N^n}\,\sum_{\beta\in\mathbb N^n}\,
\frac{1}{\gamma!}\,\overline{z}^{\alpha'}\,
\frac{\partial^{\vert\gamma\vert} \overline{W}^{j,\alpha'}}{\partial
w^\gamma}(w) \sum_{\alpha''\in\mathbb N^n}\,\sum_{\beta\in\mathbb
N^n}\, \overline{z}^{\alpha''}\,z^\beta\,
\mathcal{A}_{\alpha'',\beta,\gamma} \Big(\big\{
\overline{\Xi}_{\widehat{j},\widehat{\alpha},\widehat{\beta}}^\sim(w)
\big\}\Big)
\\
& = \sum_{\alpha\in\mathbb N^n}\,\sum_{\beta\in\mathbb N^n}\,
\overline{z}^\alpha\,z^\beta \bigg[ \sum_{\gamma\in\mathbb N^k}\,
\sum_{\alpha=\alpha'+\alpha''}\, \frac{1}{\gamma!}\,
\mathcal{A}_{\alpha'',\beta,\gamma} \Big(\big\{
\overline{\Xi}_{\widehat{j},\widehat{\alpha},\widehat{\beta}}^\sim(w)
\big\}\Big) \cdot
\frac{\partial^{\vert\gamma\vert}\overline{W}^{j,\alpha'}}{
\partial w^\gamma}(w)
\bigg].
\endaligned
\end{equation}
The computations for the second sum in~\thetag{
\ref{before-expansion-z-bar-z}} are the same:
\begin{equation}
\label{II} \footnotesize \aligned & \ \ \ \ \
-\sum_{i=1}^n\,\sum_{\alpha'\in\mathbb N^n}\,\sum_{\gamma\in\mathbb
N^k}\, \frac{1}{\gamma!}\,\overline{z}^{\alpha'}\,
\frac{\partial^{\vert\gamma\vert}\overline{Z}^{i,\alpha'}}{
\partial w^\gamma}(w)
\big(-2w+\overline{\Xi}(\overline{z},z,w)\big)^\gamma\,
\\
& = -\sum_{i=1}^n\,\sum_{\alpha'\in\mathbb
N^n}\,\sum_{\gamma\in\mathbb N^k}\,
\frac{1}{\gamma!}\,\overline{z}^{\alpha'}\,
\frac{\partial^{\vert\gamma\vert}\overline{Z}^{i,\alpha'}}{
\partial w^\gamma}(w)\,
\sum_{\alpha''\in\mathbb N^n}\,\sum_{\beta\in\mathbb N^n}\,
\overline{z}^{\alpha''}\,z^\beta\,
\mathcal{A}_{\alpha'',\beta,\gamma} \Big(\big\{
\overline{\Xi}_{\widehat{j},\widehat{\alpha},\widehat{\beta}}^\sim(w)
\big\}\Big)
\\
& = \sum_{\alpha\in\mathbb N^n}\,\sum_{\beta\in\mathbb N^n}\,
\overline{z}^\alpha\,z^\beta \bigg[
-\sum_{i=1}^n\,\sum_{\gamma\in\mathbb N^k}\,
\sum_{\alpha'+\alpha''=\alpha}\, \frac{1}{\gamma!}\,
\mathcal{A}_{\alpha'',\beta,\gamma} \Big(\big\{
\overline{\Xi}_{\widehat{j},\widehat{\alpha},\widehat{\beta}}^\sim(w)
\big\}\Big) \cdot
\frac{\partial^{\vert\gamma\vert}\overline{Z}^{i,\alpha'}}{
\partial w^\gamma}(w)
\bigg].
\endaligned
\end{equation}
The third sum in~\thetag{ \ref{before-expansion-z-bar-z}} is already
almost well written, for we indeed have, if we denote by ${\bf 0} =
(0, \dots, 0) \in \mathbb N^n$ the zero-multiindex:
\begin{equation}
\label{III} \small \aligned \sum_{\beta\in\mathbb N^n}\,
z^\beta\,W^{j,\beta}(w) = \sum_{\alpha\in\mathbb
N^n}\,\sum_{\beta\in\mathbb N^n}\, \overline{z}^\alpha\,z^\beta\,
\big[ \delta_\alpha^{\bf 0} \cdot W^{j,\beta}(w) \big],
\endaligned
\end{equation}
where $\delta_{\sf a}^{\sf b} = 0$ if ${\sf a} \neq {\sf b}$ and
equals $1$
if ${\sf a} = {\sf b}$. To transform the fourth sum in~\thetag{
\ref{before-expansion-z-bar-z}}, we must at first compute, for each
$i = 1, \dots, n$ (and for each $j = 1, \dots, k$), the first-order
partial derivatives $\frac{ \partial \overline{ \Xi}_j}{ \partial
z_i}$ from~\thetag{\ref{Xi-expansion}}, which gives, if we denote simply by ${\bf 1}_i$ the multiindex
$(0, \dots, 1, \dots, 0)$ of $\mathbb N^n$ with $1$ at the $i$-th
place and zero elsewhere:
\[
\small \aligned \frac{\partial\overline{\Xi}_j}{\partial z_i}
(\overline{z},z,w) & = \sum_{\alpha\in\mathbb
N^n}\,\sum_{\beta\in\mathbb N^n\atop \beta_i\geqslant 1}\,
\overline{z}^\alpha\,\beta_i\,z^{\beta-{\bf 1}_i}\,
\overline{\Xi}_{j,\alpha,\beta}^\sim(w)
\\
& = \sum_{\alpha\in\mathbb N^n}\,\sum_{\beta\in\mathbb N^n}\,
\overline{z}^\alpha\,z^\beta\, (\beta_i+1)\,
\overline{\Xi}_{j,\alpha,\beta+{\bf 1}_i}^\sim(w).
\endaligned
\]
Thanks to this, the fourth sum in~\thetag{
\ref{before-expansion-z-bar-z}} may be reorganized as wanted:
\begin{equation}
\label{IV} \footnotesize \aligned &
-\sum_{i=1}^n\,\sum_{\beta'\in\mathbb N^n}\,
z^{\beta'}\,Z^{i,\beta'}(w)\,
\frac{\partial\overline{\Xi}_j}{\partial z_i} (\overline{z},z,w) =
\\
& = -\sum_{i=1}^n\,\sum_{\beta'\in\mathbb N^n}\,
z^{\beta'}\,Z^{i,\beta'}(w)\, \sum_{\alpha\in\mathbb
N^n}\,\sum_{\beta''\in\mathbb N^n}\,
\overline{z}^\alpha\,z^{\beta''}\, (1+\beta_i'')\,
\overline{\Xi}_{j,\alpha,\beta''+{\bf 1}_i}^\sim(w)
\\
& = \sum_{\alpha\in\mathbb N^n}\,\sum_{\beta\in\mathbb N^n}\,
\overline{z}^\alpha\,z^\beta \bigg[ -\sum_{i=1}^n\,
\sum_{\beta'+\beta''=\beta}\, (\beta_i''+1)\,
\overline{\Xi}_{j,\alpha,\beta''+{\bf 1}_i}^\sim(w) \cdot
Z^{i,\beta'}(w) \bigg].
\endaligned
\end{equation}
Lastly, in order to transform the fifth sum in~\thetag{
\ref{before-expansion-z-bar-z}}, we must at first compute, for each
$l = 1, \dots, k$ (and for each $j = 1, \dots, k$), the first-order
partial derivatives $\frac{ \partial \overline{ \Xi}_j}{ \partial
w_l}$, and to this aim, we start by rewriting
from~\thetag{\ref{Xi-expansion}}:
\[
\overline{\Xi}_j(\overline{z},z,w) = 2w_j + \sum_{\alpha\in\mathbb
N^n}\,\sum_{\beta\in\mathbb N^n}\, \overline{z}^\alpha\,z^\beta\,
\overline{\Xi}_{j,\alpha,\beta}^\sim(w),
\]
whence it immediately follows:
\[
\aligned \frac{\partial\overline{\Xi}_j}{\partial w_l}
(\overline{z},z,w) = 2\,\delta_j^l + \sum_{\alpha\in\mathbb
N^n}\,\sum_{\beta\in\mathbb N^n}\, \overline{z}^\alpha\,z^\beta\,
\frac{\partial\overline{\Xi}_{j,\alpha,\beta}^\sim(w)}{
\partial w_l}.
\endaligned
\]
Thanks to this, the fifth sum in~\thetag{
\ref{before-expansion-z-bar-z}}, too, may be reorganized
appropriately:
\begin{equation}
\label{V} \footnotesize \aligned &
-\sum_{l=1}^k\,\sum_{\beta'\in\mathbb N^n}\,
z^{\beta'}\,W^{l,\beta'}(w)\,
\frac{\partial\overline{\Xi}_j}{\partial w_l} (\overline{z},z,w) =
\\
& = -\sum_{l=1}^k\,\sum_{\beta'\in\mathbb N^n}\,
z^{\beta'}\,W^{l,\beta'}(w)\, \bigg[ 2\,\delta_j^l +
\sum_{\alpha\in\mathbb N^n}\,\sum_{\beta''\in\mathbb N^n}\,
\overline{z}^\alpha\,z^{\beta''}\,
\frac{
\partial\overline{\Xi}_{j,\alpha,\beta''}^\sim(w)
}{
\partial w_l}
\bigg]
\\
& = \sum_{\alpha\in\mathbb N^n}\,\sum_{\beta\in\mathbb N^n}\,
\overline{z}^\alpha\,z^\beta\, \bigg[ -2\,\delta_\alpha^{\bf 0} \cdot
W^{j,\beta}(w) - \sum_{l=1}^k\, \sum_{\beta'+\beta''=\beta}\,
\frac{
\partial\overline{\Xi}_{j,\alpha,\beta''}^\sim(w)
}{
\partial w_l}\,
W^{l,\beta'}(w) \bigg].
\endaligned
\end{equation}
Summing up these five reorganized sums appearing in~\thetag{
\ref{before-expansion-z-bar-z}} as a double sum $\sum_{ \alpha} \,
\sum_\beta\, \overline{ z}^\alpha\, z^\beta \, \big( {\sf coeff}_{ j,
\alpha, \beta} \big)$, and equating to zero all the obtained
coefficients~\thetag{ \ref{I}}, \thetag{ \ref{II}}, \thetag{
\ref{III}}, \thetag{ \ref{IV}} and~\thetag{ \ref{V}}, we deduce the
following fundamental statement.

\begin{Theorem}
\label{Theorem} Let $M$ be a local generic real analytic
CR-generic submanifold of
$\mathbb C^{ n+k}$ having positive codimension $k \geqslant 1$ and
positive CR dimension $n \geqslant 1$ which is represented, in local
holomorphic coordinates $(z, w) = ( z_1, \dots, z_n, w_1, \dots,
w_k)$ centered at the origin $0 \in M$
by $k$ complex defining equations of the shape:
\[
\overline{w}_j+w_j = \overline{\Xi}_j(\overline{z},z,w) \ \ \ \ \ \ \
\ \ \ \ \ \ {\scriptstyle{(j\,=\,1\,\cdots\,k)}},
\]
with $\overline{\Xi}_j = {\rm O} ( 2)$,
and introduce the power series expansion with respect
to the variables $( \overline{ z}, z)$:
\[
-2w_j + \overline{\Xi}_j(\overline{z},z,w) =: \sum_{\alpha\in\mathbb
N^n}\,\sum_{\beta\in\mathbb N^n}\, \overline{z}^\alpha\,z^\beta\,
\overline{\Xi}_{j,\alpha,\beta}^\sim(w) \ \ \ \ \ \ \ \ \ \ \ \ \
{\scriptstyle{(j\,=\,1\,\cdots\,k)}}.
\]
For every multiindex $\alpha \in \mathbb N^n$, every multiindex
$\beta \in \mathbb N^n$ and every multiindex $\gamma \in \mathbb
N^k$, introduce also the explicit universal polynomial:
\[
\small \aligned \mathcal{A}_{\alpha,\beta,\gamma} \Big( \big\{
\overline{\Xi}_{\widehat{j},\widehat{\alpha},\widehat{\beta}}^\sim(w)
& \big\}_{\widehat{j}\in\mathbb N,\widehat{\alpha}\in\mathbb
N^n,\widehat{\beta}\in\mathbb N^n} \Big) :=
\sum_{\alpha^1+\cdots+\alpha^k=\alpha \atop
\beta_1+\cdots+\beta^k=\beta}\,
\sum_{\alpha_1^1+\cdots+\alpha_{\gamma_1}^1=\alpha^1 \atop
\beta_1^1+\cdots+\beta_{\gamma_1}^1=\beta^1}\, \cdots\,
\sum_{\alpha_1^k+\cdots+\alpha_{\gamma_k}^k=\alpha^k \atop
\beta_1^k+\cdots+\beta_{\gamma_k}^k=\beta_k}
\\
& \overline{\Xi}_{1,\alpha_1^1,\beta_1^1}^\sim(w) \cdots
\overline{\Xi}_{1,\alpha_{\gamma_1}^1,\beta_{\gamma_1}^1}^\sim(w)
\cdots\cdots\, \overline{\Xi}_{k,\alpha_1^k,\beta_1^k}^\sim(w) \cdots
\overline{\Xi}_{k,\alpha_{\gamma_k}^k,\beta_{\gamma_k}^k}^\sim(w).
\endaligned
\]
Then a general holomorphic vector field:
\[
{\sf X} = \sum_{i=1}^n\,Z^i(z,w)\, \frac{\partial}{\partial z_i} +
\sum_{l=1}^k\,W^l(z,w)\, \frac{\partial}{\partial w_l}
\]
is an infinitesimal CR-automorphism of $M$ belonging to $\mathfrak{
aut}_{CR} ( M)$, namely it has the property that $\overline{\sf X} +
{\sf X}$ is tangent to $M$ {\em if and only if}, for every $j = 1,
\dots, k$, for every $\alpha \in \mathbb N^n$ and for every $\beta
\in \mathbb N^n$, the following linear holomorphic partial
differential equation:
\begin{equation}
\label{main-formulae} \boxed{ \footnotesize \aligned 0 & \equiv
\sum_{\gamma\in\mathbb N^k}\, \sum_{\alpha=\alpha'+\alpha''}\,
\frac{1}{\gamma!}\, \mathcal{A}_{\alpha'',\beta,\gamma} \Big( \big\{
\overline{\Xi}_{\widehat{j},\widehat{\alpha},\widehat{\beta}}^\sim(w)
\big\}_{\widehat{j}\in\mathbb N,\widehat{\alpha}\in\mathbb
N^n,\widehat{\beta}\in\mathbb N^n} \Big) \cdot
\frac{\partial^{\vert\gamma\vert}\overline{W}^{j,\alpha'}}{
\partial w^\gamma}(w)
-
\\
& -\sum_{i=1}^n\,\sum_{\gamma\in\mathbb N^k}\,
\sum_{\alpha'+\alpha''=\alpha}\, \frac{1}{\gamma!}\,
\mathcal{A}_{\alpha'',\beta,\gamma} \Big(\big\{
\overline{\Xi}_{\widehat{j},\widehat{\alpha},\widehat{\beta}}^\sim(w)
\big\}\Big) \cdot
\frac{\partial^{\vert\gamma\vert}\overline{Z}^{i,\alpha'}}{
\partial w^\gamma}(w)
+
\\
& + \delta_\alpha^{\bf 0} \cdot W^{j,\beta}(w) -
\\
& -\sum_{i=1}^n\, \sum_{\beta'+\beta''=\beta}\, (\beta_i''+1)\,
\overline{\Xi}_{j,\alpha,\beta''+{\bf 1}_i}^\sim(w) \cdot
Z^{i,\beta'}(w) -
\\
& -2\,\delta_\alpha^{\bf 0} \cdot W^{j,\beta}(w) - \sum_{l=1}^k\,
\sum_{\beta'+\beta''=\beta}\,
\frac{\partial\overline{\Xi}_{j,\alpha,\beta''}^\sim(w)
}{
\partial w_l}\,
W^{l,\beta'}(w)
\endaligned}
\end{equation}
which is {\em linear} with respect to the partial derivatives:
\[
\footnotesize \aligned
\frac{\partial^{\vert\gamma\vert}Z^{i,\alpha}}{
\partial w^\gamma}(w),
\ \ \ \ \ \ \frac{\partial^{\vert\gamma'\vert}\overline{Z}^{i',\alpha'}}{
\partial w^{\gamma'}}(w),
\ \ \ \ \ \ \frac{\partial^{\vert\gamma''\vert}W^{l,\alpha''}}{
\partial w^{\gamma''}}(w),
\ \ \ \ \ \ \frac{\partial^{\vert\gamma'''\vert}\overline{W}^{l',
\alpha'''}}{
\partial w^{\gamma'''}}(w),
\endaligned
\]
together with its conjugate, are
satisfied identically in $\mathbb C\{ w\}$ by the four families of
functions:
\[
Z^{i,\alpha}(w), \ \ \ \ \ \ \overline{Z}^{i',\alpha'}(w), \ \ \ \ \ \
W^{l,\alpha''}(w), \ \ \ \ \ \ \overline{W}^{l',\alpha'''}(w).
\]
depending only upon the $k$ holomorphic variables $(w_1, \dots,
w_k)$.
\end{Theorem}

\subsection{The main strategy}
\label{mainstrategy}

According to Theorem \ref{Theorem}, finding the sought infinitesimal
CR-automorphisms $\sf X$ is equivalent to solve the linear {\sc pde}
system constructed by the equations \thetag{\ref{main-formulae}} with
the unknowns $Z^{i,\alpha}, \overline{Z}^{i',\alpha'}, W^{l,\alpha''},
\overline{W}^{l',\alpha'''}$ and afterwards,
finding the expressions of the
holomorphic coefficients $Z^i(z,w)$ and $W^l(z,w)$ of $\sf X$ for
$i=1,\ldots,n$ and $l=1,\ldots,k$, according to the formulae
\thetag{\ref{Taylor}} and thanks to the achieved solution. Then, we
can choose the following strategy to compute the desired
infinitesimal Lie algebra $\frak {aut}_{CR}(M)$ associated to a
specific real analytic CR-generic manifold $M\subset\mathbb
C^{n+k}$:

\begin{itemize}
\smallskip\item[$\bullet$] Constructing the $k$ fundamental equations
\thetag{\ref{initial-tangency-general}}.

\smallskip\item[$\bullet$] Expanding
these equations according to the formulaes \thetag{\ref{Taylor}},
\thetag{\ref{Taylor-A}} and \thetag{\ref{Xi-expansion}}.

\smallskip\item[$\bullet$] Extracting
the coefficients of each $z^\alpha\overline{z}^\beta$ for
$\alpha,\beta\in\mathbb N^n$ and constituting the linear homogeneous
{\sc pde} system of the partial differential equations
\thetag{\ref{main-formulae}}, introduced in Theorem \ref{Theorem}.

\smallskip\item[$\bullet$]
Solving the obtained system by the means of techniques in
differential algebra theory.

\smallskip\item[$\bullet$] Substituting the
solution of the {\sc pde} system into the formulae
\thetag{\ref{Taylor}} to find the holomorphic functions
$Z^{i,\alpha}(w)$ and $W^{l,\alpha} (w)$ as the coefficients of the
infinitesimal CR-automorphisms $\sf X$.
\end{itemize}

It is worth noting that the already introduced linear {\sc pde}
system defined underlying the differential ring:
\[
R
:=
\mathbb{C}
(w)\Big[
Z^{i,\alpha},
\overline{Z}^{i',\alpha'},
W^{l,\alpha''},
\overline{W}^{l',\alpha'''}
\Big]
\]
admits {\em complex} linear equations and hence at the fourth step of
the above strategy, it is not possible to employ the
Rosenfeld-Gr\"obner algorithm directly for solving it. But, it is
feasible to equip this algorithm with the conjugate operator
bar-reduction for being able to treat the complex equations.
Nevertheless, one should notice that such a linear {\sc pde} system is
not necessarily a ${\sf LinCons}$ one, since its coefficients may contain
some powers of the variables $w_l, l=1,\ldots,k$. However, a careful
look at the equations~\thetag{\ref{Xi-expansion}}
and~\thetag{\ref{main-formulae}} reveals that if the
$k$ defining functions $\overline\Xi_i$ are {\em independent} of the
variables $w_1, \dots, w_k$,
then the concerned linear {\sc pde} system
has constant complex coefficients,
namely it will be a ${\sf LinCons}$ {\sc pde} system. Such
CR-generic submanifolds $M\subset\mathbb C^{n+k}$ which are defined as the
graph of the $k$ defining complex equations:
\[
\overline w_j+w_j=\overline\Xi_j\big(\overline z,z\big)
\ \ \ \ {\scriptstyle
(j\,=\,1\,\cdots\,k)}
\]
are usually called the {\sl rigid submanifolds}. They constitute a
wide and considerably significant class of CR-generic
submanifolds ({\em see}
\cite{Baouendi,MerkerPorten,
Beloshapka1997,Boggess, 5-cubic,MS,
Shananina2000,Tumanov}). Hence, in the case of the rigid real
analytic generic CR-generic submanifolds, one can employ the
computationally
much simpler algorithm {\sl LRG} (Algorithm \ref{LRG}) to perform the
fourth step of the above strategy. In the sequel,
using this method, we compute
the Lie algebras of infinitesimal CR-automorphisms associated to
three rigid real analytic generic CR-generic submanifolds of $\mathbb
C^{1+8}$.

\section{Lie algebra of infinitesimal CR-automorphisms
\\
of CR-generic submanifolds $\mathbb M^1,\mathbb M^2$ and $\mathbb M^3$ of
$\mathbb C^{1+8}$} \label{Computations}

The aim of the current section is to compute the Lie algebras of
infinitesimal CR-automorphisms, associated to the three rigid real
analytic CR-generic submanifolds $\mathbb M^1, \mathbb M^2$ and $\mathbb
M^3$ of $\C^{ 1+8}$ of CR dimension $1$,
represented in coordinates $(z,w_1,\ldots,w_8)$
by~\thetag{\ref{models}}. We give in detail the
computations of $\frak{aut}_{CR}(\mathbb M^1)$ and since the
remaining computations of $\frak{aut}_{CR}(\mathbb M^2)$ and
$\frak{aut}_{CR}(\mathbb M^3)$ are fairly similar, then shall not
report them in detail.

For the first model $\mathbb M^1$, represented by the eight real
analytic equations:
\begin{equation}
\label{1st-model} \aligned
w^1-\overline{w}^1&=\Xi_1(z,\overline{z}):=2\,i\,z\overline{z}, \ \ \
\ \ \ \ \ \ \ \ \ \ \ \ \ \ \ \ \ \ \ \
w^2-\overline{w}^2=\Xi_2(z,\overline{z}):=2\,i\,(z^2\overline{z}+z
\overline{z}^2),
\\
w^3-\overline{w}^3&=\Xi_3(z,\overline{z}):=2\,(z^2\overline{z}-z
\overline{z}^2),
\ \ \ \ \ \
 \ \ \
w^4-\overline{w}^4=\Xi_4(z,\overline{z}):=2\,i\,(z^3\overline{z}+z
\overline{z}^3),
\\
w^5-\overline{w}^5&=\Xi_5(z,\overline{z}):=2\,(z^3\overline{z}-z
\overline{z}^3),
\ \ \ \ \ \ \ \ \
w^6-\overline{w}^6=\Xi_6(z,\overline{z}):=2\,i\,z^2\overline{z}^2,
\\
w^7-\overline{w}^7&=\Xi_7(z,\overline{z}):=2\,i\,(z^4\overline{z}+z
\overline{z}^4),
\ \ \ \ \ \ \
w^8-\overline{w}^8=\Xi_8(z,\overline{z}):=2\,(z^4\overline{z}-z
\overline{z}^4),
\endaligned
\end{equation}
a holomorphic vector field ${\sf X}:=Z(z,w)\,\partial_
z+\sum_{l=1}^8\,W^l(z,w)\,\partial_{w_l}$ is tangent to $\mathbb M^1$
if and only if the restriction of the real vector filed ${\sf X}={\sf
X}+\overline{\sf X}$ to each of the above eight defining equations
vanishes identically. In other words, if and only if the holomorphic
functions $Z$ and $W^l$, $l=1,\ldots,8$, and their conjugates enjoy
the following eight equations ({\em cf.}
\thetag{\ref{initial-tangency-general}}):
\begin{equation}
\label{par-initial-tangency-1} \footnotesize\aligned
&0\equiv\big[W^1-\overline{W}^1-2\,i\,\overline{z}\,Z-2\,i\,z
\overline{Z}\big]_{w=\overline{w}+\Xi},
\\
&0\equiv\big[W^2-\overline{W}^2-4\,i\,z\,\overline{z}Z-2\,i\,
\overline{z}^2\,Z-
2\,i\,z^2\overline{Z}-4\,i\,z\overline{z}\overline{Z}\big]_{w=
\overline{w}+\Xi},
\\
&0\equiv\big[W^3-\overline{W}^3-4\,z\overline{z}Z+2\,\overline{z}^2Z-2
\,z^2\overline{Z}+
4\,z\,\overline{z}\overline{Z}\big]_{w=\overline{w}+\Xi},
\\
&0\equiv\big[W^4-\overline{W}^4-6\,i\,z^2\overline{z}Z-2\,i
\overline{z}^3Z-2\,i\,z^3\overline{Z}-
6\,i\,z\overline{z}^2\overline{Z}\big]_{w=\overline{w}+\Xi},
\\
&0\equiv\big[W^5-\overline{W}^5-6\,z^2\overline{z}Z+2\,\overline{z}^3Z-2
\,z^3\overline{Z}+
6\,z\overline{z}^2\overline{Z}\big]_{w=\overline{w}+\Xi},
\\
&0\equiv\big[W^6-\overline{W}^6-4\,i\,z\overline{z}^2Z-4\,i\,z^2
\overline{z}\overline{Z}\big]_{w=\overline{w}+\Xi},
\\
&0\equiv\big[W^7-\overline{W}^7-8\,i\,z^3\overline{z}Z-2\,i\,
\overline{z}^4{Z}-2\,i\,z^4\overline{Z}-
8\,i\,z\overline{z}^3\overline{Z}\big]_{w=\overline{w}+\Xi},
\endaligned
\end{equation}
\begin{equation*}
\footnotesize\aligned
&0\equiv\big[W^8-\overline{W}^8-8\,z^3\overline{z}Z+2\,\overline{z}^4Z-2
\,z^4\overline{Z}+
8\,z\overline{z}^3\overline{Z}\big]_{w=\overline{w}+\Xi}.
\endaligned
\end{equation*}
These functions are real analytic and hence one may expand them with
respect to the powers of $z$ ({\em cf.} \thetag{\ref{Taylor}}):
\[
Z(z,w) = \sum_{k\in\mathbb{N}}\,z^k\,Z_k(w) \ \ \ \ \ \ \ \ \text{\rm
and} \ \ \ \ \ \ \ \ W^l(z,w) = \sum_{k\in\mathbb{N}}\,z^k\,W_k^l(w).
\]
Our current aim is to find a closed form expression for the
holomorphic functions $Z(z,w)$ and $W^l(z,w)$ by using their
corresponding Taylor series. One sees simplified expressions of these
functions in the following lemma:

\begin{Lemma}
\label{simplify} The holomorphic functions $Z(z,w)$ and $W^l(z,w),
l=1,\ldots,8$ are all polynomial with respect to $z$:
\[\footnotesize
\aligned\left[\aligned
Z(z,w)&=Z_0(w)+zZ_1(w)+z^2Z_2(w)+z^3Z_3(w)+z^4Z_4(w)+z^5Z_5(w),
\\
W^1(z,w)&=W^1_0(w)+zW^1_1(w),
\\
W^2(z,w)&=W^2_0(w)+zW^2_1(w)+z^2W^2_2(w),
\\
W^3(z,w)&=W^3_0(w)+zW^3_1(w)+z^2W^3_2(w),
\\
W^4(z,w)&=W^4_0(w)+zW^4_1(w)+z^2W^4_2(w)+z^3W^4_3(w),
\\
W^5(z,w)&=W^5_0(w)+zW^5_1(w)+z^2W^5_2(w)+z^3W^5_3(w),
\\
W^6(z,w)&=W^6_0(w),
\\
W^7(z,w)&=W^7_0(w)+zW^7_1(w)+z^2W^7_2(w)+z^3W^7_3(w)+z^4W^7_4(w),
\\
W^8(z,w)&=W^8_0(w)+zW^8_1(w)+z^2W^8_2(w)+z^3W^8_3(w)+z^4W^8_4(w).
\endaligned
\right.
\endaligned
\]
\end{Lemma}

\proof After expansion, the first equation
\thetag{\ref{par-initial-tangency-1}} reads:
\begin{equation}
\aligned 0&\equiv \sum_{k\in\mathbb
N}\,z^k\,\big[W^1_k(\overline{w}_1+\Xi_1,\overline{w}_2+\Xi_2,\ldots,
\overline{w}_8+\Xi_8)-
\\
&\ \ \ \ \ \ \ \ \ \ \ \ -2\,i\,\overline
z\,Z_k(\overline{w}_1+\Xi_1,\overline{w}_2+\Xi_2,\ldots,\overline{w}_8
+\Xi_8)\big]+
\\
&+\sum_{k\in\mathbb N}\,\overline z^k\big[-\overline W^1_k(\overline
w_1,\overline w_2,\ldots,\overline w_8)-2\,i\,z\,\overline
Z_k(\overline w_1,\overline w_2,\ldots,\overline w_8)\big].
\endaligned
\end{equation}
Then, we expand further each holomorphic function $Z^k$ and $W^1_k$
according to the formulae ({\em cf.} \thetag{\ref{Taylor-A}}):
\[
\footnotesize\aligned
A&(\overline{w}_1+\Xi_1,\overline{w}_2+\Xi_2,\ldots,\overline{w}_8
+\Xi_8)=
\\
&\sum_{l_1,\ldots,l_8\in\mathbb N}\,A_{w_1^{l_1}\,w_2^{l_2}\ldots
w_8^{l_8}}(\overline w_1,\overline w_2,\ldots,\overline
w_8)\,\frac{(2\,i\,z\overline
z)^{l_1}}{l_1!}\cdot\frac{(2\,i\,z^2\overline z+2\,i\,z\overline
z^2)^{l_2}}{l_2!}\cdot\frac{(2\,z^2\overline z-2\,z\overline
z^2)^{l_3}}{l_3!}\cdot
\\
& \ \ \ \ \ \ \ \ \ \ \ \ \ \ \ \ \ \ \ \ \ \ \ \ \ \
 \ \frac{(2\,i\,z^3\overline z+2\,i\,z\overline
z^3)^{l_4}}{l_4!}\cdot\frac{(2\,z^3\overline z+2\,z\overline
z^3)^{l_5}}{l_5!}\cdot\frac{(2\,i\,z^2\overline
z^2)^{l_6}}{l_6!}\cdot
\\
& \ \ \ \ \ \ \ \ \ \ \ \ \ \ \ \ \ \ \ \ \ \ \ \ \ \ \ \ \ \ \ \ \ \
\ \ \ \ \ \ \ \ \ \ \frac{(2\,i\,z^4\overline z+2\,i\,z\overline
z^4)^{l_7}}{l_7!}\cdot\frac{(2\,z^4\overline z+2\,z\overline
z^4)^{l_8}}{l_8!}.
\endaligned
\]
 Chasing the coefficient of
$\overline{ z}^k$ for every $k \geqslant 2$
after further expansion, we therefore see that the first two lines
give absolutely no contribution, and that from the third line, it
only comes: $0 \equiv \overline{ W}_k^1 ( \overline{ w})$, which is
what was claimed about the $W_k^1$.

Next, chasing the coefficient of
$z \overline{ z}^{ k'}$ for every $k' \geqslant 6$, we get $0 \equiv
\overline{ Z}_{ k'} ( \overline{ w})$. The seven remaining families
of vanishing equations $0 \equiv \overline{W}_{k_l}^l ( \overline{
w})$ for $l=2,\ldots,8$ with $k_2,k_3\geqslant 3$, with $k_4,k_5\geqslant 4,
k_6\geqslant 1$ and with $k_7,k_8\geqslant 5$ are obtained in a completely
similar way by looking at the second to eighth equations of
\thetag{\ref{par-initial-tangency-1}}.
\endproof

After this extensive simplification, we try to determine expressions
of the remaining $33$ holomorphic functions $Z_0,Z_1,\ldots,W^8_4$ in
the above lemma. Substituting in
\thetag{\ref{par-initial-tangency-1}} the already obtained
expressions for the nine functions $Z,W^l, l=1,\dots,8$, the
fundamental equations \thetag{\ref{par-initial-tangency-1}} change
into the form:
\begin{equation}\footnotesize
\label{W1} \aligned 0&\equiv
\sum_{t=0}^1\,\big(z^t\,W^1_t(\overline{w}+\Xi)\big)-\sum_{t=0}^1\,
\big(\overline{z}^t\,\overline{W}^1_t(\overline{w})\big)
-2\,i\sum_{t=0}^5\,\big(z^t\overline{z}\,Z(\overline{w}+\Xi)\big)-2\,i
\sum_{t=0}^5\,\big(z\overline{z}^t\,\overline{Z}_t(\overline{w})\big),
\endaligned
\end{equation}
\begin{equation}
\label{W2} \footnotesize\aligned
0&\equiv\sum_{t=0}^2\,\big(z^t
\,W^2_t(\overline{w}+\Xi)\big)-\sum_{t=0}^2\,\big(\overline{z}^t\,
\overline{W}^2_t(\overline{w})\big)-
4\,i\sum_{t=0}^5\,\big(z^{t+1}\overline{z}\,Z(\overline{w}+\Xi)\big)-
\\
&-2\,i\sum_{t=0}^5\big(z^t\overline{z}^2\,Z_t(\overline{w}+\Xi)\big)
-2\,i\sum_{t=0}^5\,\big(z^2\overline{z}^t\,
\overline{Z}_t(\overline{w})\big)-4\,i\sum_{t=0}^5\,\big(z
\overline{z}^{t+1}\overline{Z}_t(\overline{w})\big),
\ \ \ \ \ \ \ \ \ \ \ \ \ \ \ \ \ \ \ \ \ \ \ \ \ \ \ \ \ \ \ \ \ \ \
\ \  \
\endaligned
\end{equation}
\begin{equation}
\label{W3} \footnotesize\aligned
0&\equiv\sum_{t=0}^2\,\big(z^t
\,W^3_t(\overline{w}+\Xi)\big)-\sum_{t=0}^2\,\big(\overline{z}^t\,
\overline{W}^3_t(\overline{w})\big)-
4\sum_{t=0}^5\,\big(z^{t+1}\overline{z}\,Z(\overline{w}+\Xi)\big)+
\\
&+ 2\sum_{t=0}^5\big(z^t\overline{z}^2\,Z_t(\overline{w}+\Xi)\big)
-2\sum_{t=0}^5\,\big(z^2\overline{z}^t\,
\overline{Z}_t(\overline{w})\big)+4\sum_{t=0}^5\,\big(z\overline{z}^{t
+1}\overline{Z}_t(\overline{w})\big),
\ \ \ \ \ \ \ \ \ \ \ \ \ \ \ \ \ \ \ \ \ \ \ \ \ \ \ \ \ \ \ \ \ \ \
\ \ \
\endaligned
\end{equation}
\begin{equation}
\label{W4} \footnotesize\aligned
0&\equiv\sum_{t=0}^3\,\big(z^t
\,W^4_t(\overline{w}+\Xi)\big)-\sum_{t=0}^3\,\big(\overline{z}^t\,
\overline{W}^4_t(\overline{w})\big)-
6\,i\sum_{t=0}^5\,\big(z^{t+2}\overline{z}\,Z(\overline{w}+\Xi)\big)-
\\
&- 2\,i\sum_{t=0}^5\big(z^t\overline{z}^3\,Z_t(\overline{w}+\Xi)\big)
-2\,i\sum_{t=0}^5\,\big(z^3\overline{z}^t\,
\overline{Z}_t(\overline{w})\big)-6\,i\sum_{t=0}^5\,\big(z
\overline{z}^{t+2}\overline{Z}_t(\overline{w})\big),
\ \ \ \ \ \ \ \ \ \ \ \ \ \ \ \ \ \ \ \ \ \ \ \ \ \ \ \ \ \ \ \ \ \ \
\ \ \
\endaligned
\end{equation}
\begin{equation}
\label{W5} \footnotesize\aligned
0&\equiv\sum_{t=0}^3\,\big(z^t
\,W^5_t(\overline{w}+\Xi)\big)-\sum_{t=0}^3\,\big(\overline{z}^t\,
\overline{W}^5_t(\overline{w})\big)-
6\sum_{t=0}^5\,\big(z^{t+2}\overline{z}\,Z(\overline{w}+\Xi)\big)+
\\
&+ 2\sum_{t=0}^5\big(z^t\overline{z}^3\,Z_t(\overline{w}+\Xi)\big)
-2\sum_{t=0}^5\,\big(z^3\overline{z}^t\,
\overline{Z}_t(\overline{w})\big)+6\sum_{t=0}^5\,\big(z\overline{z}^{t
+2}\overline{Z}_t(\overline{w})\big),
\ \ \ \ \ \ \ \ \ \ \ \ \ \ \ \ \ \ \ \ \ \ \ \ \ \ \ \ \ \ \ \ \ \ \
\ \ \
\endaligned
\end{equation}
\begin{equation}
\label{W6} \footnotesize\aligned 0&\equiv
W^6_0(\overline{w}+\Xi)-\overline{W}^6_t(\overline{w})-
4\,i\sum_{t=0}^5\,\big(z^{t+1}\overline{z}^2\,Z(\overline{w}+\Xi)\big)-
4\,i\sum_{t=0}^5\big(z^2\overline{z}^{t+1}\,
\overline{Z}_t(\overline{w})\big),
\endaligned
\end{equation}
\begin{equation}
\label{W7} \footnotesize\aligned
0&\equiv\sum_{t=0}^4\,\big(z^t
\,W^7_t(\overline{w}+\Xi)\big)-\sum_{t=0}^4\,\big(\overline{z}^t\,
\overline{W}^7_t(\overline{w})\big)-
8\,i\sum_{t=0}^5\,\big(z^{t+3}\overline{z}\,Z(\overline{w}+\Xi)\big)-
\\
&- 2\,i\sum_{t=0}^5\big(z^t\overline{z}^4\,Z_t(\overline{w}+\Xi)\big)
-2\,i\sum_{t=0}^5\,\big(z^4\overline{z}^t\,
\overline{Z}_t(\overline{w})\big)-8\,i\sum_{t=0}^5\,\big(z
\overline{z}^{t+3}\overline{Z}_t(\overline{w})\big),
\ \ \ \ \ \ \ \ \ \ \ \ \ \ \ \ \ \ \ \ \ \ \ \ \ \ \ \ \ \ \ \ \ \ \
\ \ \
\endaligned
\end{equation}
\begin{equation}
\label{W8} \footnotesize\aligned
0&\equiv\sum_{t=0}^4\,\big(z^t
\,W^8_t(\overline{w}+\Xi)\big)-\sum_{t=0}^4\,\big(\overline{z}^t\,
\overline{W}^8_t(\overline{w})\big)-
8\sum_{t=0}^5\,\big(z^{t+3}\overline{z}\,Z(\overline{w}+\Xi)\big)+
\\
&+ 2\sum_{t=0}^5\big(z^t\overline{z}^4\,Z_t(\overline{w}+\Xi)\big)
-2\sum_{t=0}^5\,\big(z^4\overline{z}^t\,
\overline{Z}_t(\overline{w})\big)+8\sum_{t=0}^5\,\big(z\overline{z}^{t
+3}\overline{Z}_t(\overline{w})\big).
\ \ \ \ \ \ \ \ \ \ \ \ \ \ \ \ \ \ \ \ \ \ \ \ \ \ \ \ \ \ \ \ \ \ \
\ \ \
\endaligned
\end{equation}
Now we arrive at the point where we should
expand the functions $Z_\bullet(\overline{w}+\Xi)$ and
$W^\bullet_\bullet(\overline{w}+\Xi)$ appearing in the above eight
equations \thetag{\ref{W1}}--\thetag{\ref{W8}} using the Taylor
series~\thetag{\ref{Taylor-A}}. Afterwards, the
coefficients of $z^\mu\overline{z}^\nu$ in these equations should be
extracted and set equal to zero,
identically. This is equivalent to solve the
${\sf LinCons}$ {\sc pde} system constructed by these coefficients
underlying the differential ring $R:=\mathbb
C(w)[Z_\bullet,W^1_\bullet,\ldots,W^8_\bullet]$, equipped with the
functional conjugation operator.

Doing so, we have extracted these coefficients
as much as it was needed\footnote{ The number of the equations that
we found during this step was more than 63 but many of them could be
obtained from the remaining ones.}
and obtained the following
system of 63 equations with 33 unknowns $Z_0,Z_1,\ldots,W^8_4$,
which are in fact the
indeterminates of the differential ring $R$. Here by $(\mu,\nu:\ell)$
we mean the coefficient of $z^\mu\overline{z}^\nu$ in the fundamental
equation ($\ell$) for $\ell=\ref{W1},\ldots,\ref{W8}$:
\begin{equation}
\label{PDE-sys}\footnotesize\aligned &(0,1:\ref{W1}): \,
-2\,i\,Z_0-\overline{W}^1_1\equiv 0, \ \ \ \ \ (0,1:\ref{W2}): \
\overline{W}^2_1\equiv 0, \ \ \ \ \ (0,1:\ref{W3}): \
\overline{W}^3_1\equiv 0, \ \ \ \ \ (0,1:\ref{W4}): \
\overline{W}^4_1\equiv 0,
\\
& (0,1:\ref{W5}): \ \overline{W}^5_1\equiv 0, \ \ \ \ \
(0,1:\ref{W6}): \ \overline{W}^6_1\equiv 0, \ \ \ \ \ (0,1:\ref{W7}):
\ \overline{W}^7_1\equiv 0, \ \ \ \ \ (0,1:\ref{W8}): \
\overline{W}^8_1\equiv 0,
\\
& (1,1:\ref{W1}): \
-2\,i\,Z_1-2\,i\,\overline{Z}_1+2\,i\,\overline{W}^1_{0,w^1}\equiv 0,
\ \ \ \ \ (1,1:\ref{W2}): \
-4\,i\,Z_0-4\,i\,\overline{Z}_0+2\,i\,\overline{W}^2_{0,w^1}\equiv 0,
\\
& (1,1:\ref{W3}): \
-4\,Z_0+4\,\overline{Z}_0+2\,i\,\overline{W}^3_{0,w^1}\equiv 0, \ \ \
\ \ (1,1:\ref{W4}): \ W^4_{0,w^1}\equiv 0, \ \ \ \ \  (1,1:\ref{W5}):
\ W^5_{0,w^1}\equiv 0,
\\
&(1,1:\ref{W6}): \ W^6_{0,w^1}\equiv 0, \ \ \ \ \ (1,1:\ref{W7}): \
W^7_{0,w^1}\equiv 0, \ \ \ \ \ (1,1:\ref{W8}): \ W^8_{0,w^1}\equiv 0,
\\
& (0,2:\ref{W2}): \ 2\,i\,Z_0+\overline{W}^2_2\equiv 0, \ \ \ \ \
(0,2:\ref{W3}): \ 2\,Z_0-\overline{W}^3_2\equiv 0, \ \ \ \ \
(0,2:\ref{W4}): \ \overline{W}^4_2\equiv 0, \ \ \ \ \
\\
& (0,2:\ref{W5}): \ \overline{W}^5_2\equiv 0, \ \ \ \ \
(0,2:\ref{W7}): \ \overline{W}^7_2\equiv 0, \ \ \ \ \ (0,2:\ref{W8}):
\ \overline{W}^8_2\equiv 0,
\\
& (1,2:\ref{W1}): \
-2\,i\,\overline{Z}_2+2\,i\,W^1_{0,w^2}-2\,W^1_{0,w^3}+4
\,Z_{0,w^1}\equiv
0, \ \ \ \ \ (1,2:\ref{W2}): \
-2\,i\,Z_1-4\,i\,\overline{Z}_1-2\,W^2_{0,w^3}+2\,i\,W^2_{0,w^2}\equiv
0,
\\
& (1,2:\ref{W3}): \
2\,Z_1+4\,\overline{Z}_1-2\,W^3_{0,w^3}+2\,i\,W^3_{0,w^2}\equiv 0, \
\ \ \ \ (1,2:\ref{W4}): \
-6\,i\,\overline{Z}_0-2\,W^4_{0,w^3}+2\,i\,W^4_{0,w^2}\equiv 0,
\\
& (1,2:\ref{W5}): \
6\,\overline{Z}_0-2\,W^5_{0,w^3}+2\,i\,W^5_{0,w^2}\equiv 0, \ \ \ \ \
(1,2:\ref{W6}): \ -4\,i\,{Z}_0-2\,W^6_{0,w^3}+2\,i\,W^6_{0,w^2}\equiv
0,
\\
&(1,2:\ref{W7}): \ -2\,W^7_{0,w^3}+2\,i\,W^7_{0,w^2}\equiv 0, \ \ \ \
\ (1,2:\ref{W8}): \ -2\,W^8_{0,w^3}+2\,i\,W^8_{0,w^2}\equiv 0,
\\
&(0,3:\ref{W4}): \ 2\,i\,Z_0 +\overline{W}^4_3\equiv 0, \ \ \ \ \
(0,3:\ref{W5}): \ 2\,Z_0 -\overline{W}^5_3\equiv 0, \ \ \ \ \
(0,3:\ref{W7}): \ \overline{W}^7_3\equiv 0, \ \ \ \ \ (0,3:\ref{W8}):
\ \overline{W}^8_3\equiv 0,
\\
& (1,3:\ref{W1}): \
-\overline{W}^1_{0,w^5}-i\,\overline{Z}_3+i\,W^1_{0,w^4}\equiv 0, \ \
\ \ \ \ (1,3:\ref{W2}): \
-2\,i\,\overline{Z}_2+i\,W^2_{0,w^4}-\overline{W}^2_{0,w^5}\equiv 0,
\\
& (1,3:\ref{W3}): \
2\,\overline{Z}_2-W^3_{0,w^5}+i\,\overline{W}^3_{0,w^4}\equiv 0, \ \
\ \ \ (1,3:\ref{W4}): \
-2\,i\,Z_1-6\,i\,\overline{Z}_1-2\,W^4_{0,w^5}+2\,i\,
\overline{W}^4_{0,w^4}\equiv
0,
\\
& (1,3:\ref{W5}): \
2\,Z_1+6\,\overline{Z}_1-2\,W^5_{0,w^5}+2\,i\,
\overline{W}^5_{0,w^4}\equiv
0, \ \ \ \ \ (1,3:\ref{W6}): \
-2\,W^6_{0,w^5}+2\,i\,\overline{W}^6_{0,w^4}\equiv 0,
\\
&
(1,3:\ref{W7}): \
-8\,i\,\overline{Z}_0-2\,W^7_{0,w^5}+2\,i\,\overline{W}^7_{0,w^4}\equiv
0, \ \ \ \ \ (1,3:\ref{W8}): \
8\,\overline{Z}_0-2\,W^8_{0,w^5}+2\,i\,\overline{W}^8_{0,w^4}\equiv
0,
\\
 & (2,2:\ref{W1}): \
i\,W^1_{0,w^6}-Z_{1,w^1}\equiv 0, \ \ \ \ \ (2,2:\ref{W2}): \
i\,W^2_{0,w^6}-i\,Z_2-i\,\overline{Z}_2\equiv 0,
\\
&(2,2:\ref{W3}): \ -W^3_{0,w^6}+i\,Z_2-i\,\overline{Z}_2\equiv 0, \ \
\ \ \ (2,2:\ref{W4}): \ i\,W^4_{0,w^6}\equiv 0, \ \ \ \ \
(2,2:\ref{W5}): \ i\,W^5_{0,w^6}\equiv 0,
\\
& (2,2:\ref{W6}): \
i\,W^6_{0,w^6}-2\,i\,Z_1-2\,i\,\overline{Z}_1\equiv 0, \ \ \ \ \
(2,2:\ref{W7}): \ i\,W^7_{0,w^6}\equiv 0, \ \ \ \ \ (2,2:\ref{W8}): \
i\,W^8_{0,w^6}\equiv 0,
\\
& (0,4:\ref{W7}): \ 2\,i\,Z_0+\overline{W}^7_4\equiv 0, \ \ \ \ \
(0,4:\ref{W8}): \ 2\,Z_0-\overline{W}^8_4\equiv 0,
\\
& (1,4:\ref{W1}): \
-i\,\overline{Z}_4+i\,\overline{W}^1_{0,w^7}-W^1_{0,w^8}\equiv 0, \ \
\ \ \ (1,4:\ref{W2}): \
-2\,i\,\overline{Z}_3+i\,\overline{W}^2_{0,w^7}-W^2_{0,w^8}\equiv 0,
\\
& (1,4:\ref{W3}): \ {W}^3_{0,w^8}-i\,W^3_{0,w^7}\equiv 0, \ \ \ \ \
(1,4:\ref{W4}): \
-3\,i\,\overline{Z}_2+i\,\overline{W}^4_{0,w^7}-W^4_{0,w^8}\equiv 0,
\\
& (1,4:\ref{W5}): \
3\,\overline{Z}_2-\overline{W}^5_{0,w^8}+i\,W^5_{0,w^7}\equiv 0, \ \
\ \ \ (1,4:\ref{W6}): \
2\,i\,\overline{W}^6_{0,w^7}-2\,i\,W^6_{0,w^8}\equiv 0,
\\
& (1,4:\ref{W7}): \
-2\,i\,Z_1-8\,i\,\overline{Z}_1-2\,\overline{W}^7_{0,w^8}+2\,i
\,W^7_{0,w^7}\equiv
0, \ \ \ \ \ (1,4:\ref{W8}): \
2\,i\,Z_1+8\,\overline{Z}_1-2\,\overline{W}^8_{0,w^8}+2\,i
\,W^8_{0,w^7}\equiv
0,
\\
 & (1,4:\ref{W2}): \ -i\,\overline{Z}_3\equiv 0, \ \ \ \ \
(1,4:\ref{W4}): \ -i\,{Z}_2\equiv 0,  \ \ \ \ \
 \ \ \ \ \ \ \ \ \ \ \ \ \ \ \ \ \ \ \ \ \ \ \ \ \ \ \ \ \ \ \ \ \ \
 \ \ \ \ \ \ \ \ \ \ \ \ \ \ \
\\
& (1,5:\ref{W1}): \ -i\,\overline{Z}_5\equiv 0, \ \ \ \ \
(1,5:\ref{W2}): \ -i\,\overline{Z}_4\equiv 0.
\endaligned
\end{equation}

We employ the {\sc Maple} package {\tt DifferentialAlgebra},
performing
the approach described in Section \ref{Solving PDE} based on the
Rosenfeld-Gr\"obner algorithm to solve such a
${\sf LinCons}$ {\sc pde} system, and we obtain the following
general solution:
\begin{equation*}
\footnotesize\aligned Z(z,w)&:=\underbrace{{\sf c}+i\,{\sf
d}}_{Z_0(w)} +\underbrace{({\sf a}+i\,{\sf b})}_{Z_1(w)}z,
\\
W^1(z,w)&:=\underbrace{{\sf c}_1+2\,{\sf
a}\,w^1}_{W^1_0(w)}+2\,(\underbrace{{\sf d}+i\,{\sf
c}}_{W^1_1(w)})\,z,
\\
 W^2(z,w)&:=\underbrace{{\sf c}_2+4\,{\sf
c}\,w^1+3\,{\sf a}\,w^2-{\sf b}\,w^3}_{W^2_0(w)}+\underbrace{2\,({\sf
d}+i\,{\sf c})}_{W^2_2(w)}\,z^2,
\\
 W^3(z,w)&:=\underbrace{{\sf c}_3+4\,{\sf
d}\,w^1+{\sf b}\,w^2+3\,{\sf a}\,w^3}_{W^3_0(w)}+\underbrace{2\,({\sf
c}-i\,{\sf d})}_{W^3_2(w)}\,z^2,
\\
 W^4(z,w)&:=\underbrace{{\sf
c}_4+3\,{\sf c}\,w^2-3\,{\sf d}\,w^3+4\,{\sf a}\,w^4-2\,{\sf
b}\,w^5}_{W^4_0(w)}+\underbrace{2\,({\sf d}+i\,{\sf
c})}_{W^4_3(w)}\,z^3,
\\
 W^5(z,w)&:=\underbrace{{\sf c}_5+3\,{\sf
d}\,w^2+3\,{\sf c}\,w^3+2\,{\sf b}\,w^4+4\,{\sf
a}\,w^5}_{W^5_0(w)}+\underbrace{2\,({\sf c}-i\,{\sf
d})}_{W^5_3(w)}\,z^3,
\\
 W^6(z,w)&:=\underbrace{{\sf c}_6+2\,{\sf
c}\,w^2+2\,{\sf d}\,w^3+4\,{\sf a}\,w^6}_{W^6_0(w)},
\endaligned
\end{equation*}
\begin{equation*}
\footnotesize\aligned W^7(z,w)&:=\underbrace{{\sf c}_7+4\,{\sf
c}\,w^4-4\,{\sf d}\,w^5+5\,{\sf a}\,w^7-3\,{\sf
b}\,w^8}_{W^7_0(w)}+\underbrace{2\,({\sf d}+i\,{\sf
c})}_{W^7_4(w)}\,z^4,
\\
W^8(z,w)&:=\underbrace{{\sf c}_8+4\,{\sf d}\,w^4+4\,{\sf
c}\,w^5+3\,{\sf b}\,w^7+5\,{\sf
a}\,w^8}_{W^8_0(w)}+\underbrace{2\,({\sf c}-i\,{\sf
d})}_{W^8_4(w)}\,z^4,
\endaligned
\end{equation*}
for some twelve real variables ${\sf c}_1, \ldots, {\sf c}_8, {\sf
a}, {\sf b}, {\sf c}, {\sf d} \in \R$.
Putting these functions into the
general expression ${\sf
X}:=Z(z,w)\,\partial_z+\sum_{l=1}^8\,W^l(z,w)\,\partial_{w^l}$ of the
desired infinitesimal CR-automorphisms gives their general form. One
can check easily that such a parametrized
holomorphic vector field enjoys the
tangency condition $({\sf X}+\overline{\sf X})|_{\mathbb M^1}\equiv
0$. Picking the coefficients of the above twelve real variables in
this general expression, provides twelve $\mathbb R$-linearly
independent infinitesimal CR-automorphisms ${\sf X}_1,\ldots,{\sf
X}_{12}$ which constitute a basis for the desired Lie algebra
$\frak{aut}_{CR}(\mathbb M^1)$. Before presenting them, we recall
that a Lie algebra $\frak g$ is called {\sl graded} in the sense of
Tanaka whenever it admits a gradation like:
\[
\frak g:=\frak g_{-k}\oplus\cdots\oplus\frak g_{-1}\oplus \frak
g_0\oplus\frak g_1\oplus\cdots\oplus\frak g_{k'}, \ \ \ \ \ \ \ \ \
{\scriptstyle{(k,\,k'\,\in\,\mathbb{N})}}
\]
satisfying:
\[
[\frak g_i,\frak g_j]\subset\frak g_{i+j}.
\]

\begin{Theorem}
\label{Theorem-1}
The Lie algebra of infinitesimal CR-automorphisms
$\frak{aut}_{CR}(\mathbb M^1)$ of the rigid real analytic
CR-generic
submanifold $\mathbb M^1\subset\mathbb C^{1+8}$  is $12$-dimensional,
generated by the twelve $\mathbb R$-linearly independent holomorphic
vector fields:
\begin{equation*}
\footnotesize \left\{ \aligned
 {\sf X}_i&:=\partial_{w_i}, \ \ \ \ \ {
i=1,\ldots,8},
\\
 {\sf X}_9&:=z\partial_{z}+2\,w^1\partial_{w^1}+3\,w^2\partial_{w^2}+3
\,w^3\partial_{w^3}+4\,w^4\partial_{w^4}+
 \\
&+
 4\,w^5\partial_{w^5}+4\,w^6\partial_{w^6}+5\,w^7\partial_{w^7}+5\,w^8
\partial_{w^8},
 \\
 {\sf
 X}_{10}&:=i\,z\partial_{z}-w^3\partial_{w^2}+w^2\partial_{w^3}-2\,w^5
\partial_{w^4}+2\,w^4\partial_{w^5}-3\,w^8\partial_{w^7}+3\,w^7
\partial_{w^8},
 \\
 {\sf
 X}_{11}&:=\partial_z+2\,i\,z\partial_{w^1}+(4\,w^1+2\,i
\,z^2)\partial_{w^2}+2\,z^2\partial_{w^3}+(3\,w^2+2\,i
\,z^3)\partial_{w^4}+
 \\
 &+
 (3\,w^3+2\,z^3)\partial_{w^5}+2\,w^2\partial_{w^6}+(4\,w^4+2\,i
\,z^4)\partial_{w^7}+(4\,w^5+2\,z^4)\partial_{w^8},
 \\
 {\sf
 X}_{12}&:=i\,\partial_{z}+2\,z\partial_{w^1}+2\,z^2\partial_{w^2}+(4
\,w^1+-2\,i\,z^2)\partial_{w^3}+(-3\,w^3+2\,z^3)\partial_{w^4}+
 \\
 &+
 (3\,w^2-2\,i\,z^3)\partial_{w^5}+2\,w^3\partial_{w^6}+(-4\,w^5+2
\,z^4)\partial_{w^7}+(4\,w^4-2\,i\,z^4)\partial_{w^8}.
 \endaligned
 \right.
\end{equation*}
Furthermore, it is graded of the form:
\[
\frak{aut}_{CR}(\mathbb M^1):=\frak g_{-5}\oplus\frak
g_{-4}\oplus\frak g_{-3}\oplus\frak g_{-2}\oplus\frak
g_{-1}\oplus\frak g_0
\]
with $\frak g_{-5}:=\langle{\sf X}_7,{\sf X}_8\rangle$,with $\frak
g_{-4}:=\langle{\sf X}_4,{\sf X}_5,{\sf X}_6\rangle$ ,with $\frak
g_{-3}:=\langle{\sf X}_2,{\sf X}_3\rangle$, with $\frak
g_{-2}:=\langle{\sf X}_1\rangle$, with $\frak g_{-1}:=\langle{\sf
X}_{11},{\sf X}_{12}\rangle$, and with $\frak g_0:=\langle{\sf
X}_9,{\sf X}_{10}\rangle$ and together with the following table of
commutators:

\medskip
\begin{center}
\begin{tabular}{|c|c|c|c|c|c|c|c|c|c|c|c|c|}
  \hline
   & ${\sf X}_{1}$ & ${\sf X}_{2}$ & ${\sf X}_{3}$ & ${\sf X}_{4}$ &
${\sf X}_{5}$ & ${\sf X}_{6}$ & ${\sf X}_{7}$ &
    ${\sf X}_{8}$ & ${\sf X}_{9}$ & ${\sf X}_{10}$ & ${\sf X}_{11}$ &
${\sf X}_{12}$ \\
  \hline
  ${\sf X}_{1}$ & $0$ & $0$ & $0$ & $0$ & $0$ & $0$ & $0$ & $0$ & $2{\sf
X}_{1}$ & $0$ & $4{\sf X}_{2}$ & $4{\sf X}_{3}$ \\
  ${\sf X}_{2}$ & $\ast$ & $0$ & $0$ & $0$ & $0$ & $0$ & $0$ & $0$ &
$3{\sf X}_{2}$ & ${\sf X}_{3}$ & $3{\sf X}_{4}$+$2{\sf X}_{6}$ & $3{\sf
X}_{5}$ \\
  ${\sf X}_{3}$ & $\ast$ & $\ast$ & $0$ & $0$ & $0$ & $0$ & $0$ & $0$ &
$3{\sf X}_{3}$ & -${\sf X}_{2}$ & $3{\sf X}_{5}$ & $-3{\sf
X}_{4}$+$2{\sf X}_{6}$ \\
  ${\sf X}_{4}$ & $\ast$ & $\ast$ & $\ast$ & $0$ & $0$ & $0$ & $0$ & $0$
& $4{\sf X}_{4}$ & $2{\sf X}_{5}$ & $4{\sf X}_{7}$ & $4{\sf X}_{8}$ \\
  ${\sf X}_{5}$ & $\ast$ & $\ast$ & $\ast$ & $\ast$ & $0$ & $0$ & $0$ &
$0$ & $4{\sf X}_{5}$ & $-2{\sf X}_{4}$ & $4{\sf X}_{8}$ & $-4{\sf
X}_{7}$ \\
  ${\sf X}_{6}$ & $\ast$ & $\ast$ & $\ast$ & $\ast$ & $\ast$ & $0$ & $0$
& $0$ & $4{\sf X}_{6}$ & $0$ & $0$ & $0$ \\
  ${\sf X}_{7}$ & $\ast$ & $\ast$ & $\ast$ & $\ast$ & $\ast$ & $\ast$ &
$0$ & $0$ & $5{\sf X}_{7}$ & $3{\sf X}_{8}$ & $0$ & $0$ \\
  ${\sf X}_{8}$ & $\ast$ & $\ast$ & $\ast$ & $\ast$ & $\ast$ & $\ast$ &
$\ast$ & $0$ & $5{\sf X}_{8}$ & $-3{\sf X}_{7}$ & $0$ & $0$ \\
  ${\sf X}_{9}$ & $\ast$ & $\ast$ & $\ast$ & $\ast$ & $\ast$ & $\ast$ &
$\ast$ & $\ast$ & $0$ & $0$ & -${\sf X}_{11}$ & -${\sf X}_{12}$ \\
  ${\sf X}_{10}$ & $\ast$ & $\ast$ & $\ast$ & $\ast$ & $\ast$ & $\ast$ &
$\ast$ & $\ast$ & $\ast$ & $0$ & -${\sf X}_{12}$ & ${\sf X}_{11}$ \\
  ${\sf X}_{11}$ & $\ast$ & $\ast$ & $\ast$ & $\ast$ & $\ast$ & $\ast$ &
$\ast$ & $\ast$ & $\ast$ & $\ast$ & $0$ & $4{\sf X}_1$ \\
  ${\sf X}_{12}$ & $\ast$ & $\ast$ & $\ast$ & $\ast$ & $\ast$ & $\ast$ &
$\ast$ & $\ast$ & $\ast$ & $\ast$ & $\ast$ & $0$ \\
  \hline
\end{tabular}
\end{center}
\end{Theorem}

\subsection{Two remaining sought Lie algebras $\frak{aut}_{CR}(\mathbb
M^2)$ and $\frak{aut}_{CR}(\mathbb M^3)$}
One can perform long computations similar to what we did for
$\frak{aut}_{CR}(\mathbb M^1)$ and obtain the structure of
the two remaining Lie algebras $\frak{aut}_{CR}(\mathbb M^2)$ and
$\frak{aut}_{CR}(\mathbb M^3)$. Here, we omit the corresponding
intermediate computations\,\,---\,\,since they are similar to those
of $\mathbb M^1$ and offer no new aspect.

\begin{Theorem}
\label{Theorem-2}
 The Lie algebra of infinitesimal CR-automorphisms
$\frak{aut}_{CR}(\mathbb M^2)$ of the rigid real analytic generic
CR-generic submanifold
$\mathbb M^2\subset\mathbb C^{1+8}$, represented as the
graph of the eight defining equations:
\[\aligned
w^j-\overline{w}^j&=\Xi_j(z,\overline{z}), \ \ \ \ \ \ \ \ \ \ \ \
{\scriptstyle (j\,=\,1\,\cdots\,6)},
\\
w^7-\overline{w}^7&=\Xi_7'(z,\overline{z}):=2\,i\,(z^3\overline{z}^2+z^2
\overline{z}^3),
\\
w^8-\overline{w}^8&=\Xi_8'(z,\overline{z}):=2\,(z^3\overline{z}^2-z^2
\overline{z}^3),
\endaligned
\]
 is $12$-dimensional with the holomorphic coefficients:
\begin{equation*}
 \footnotesize\aligned Z(z,w)&:=\underbrace{{\sf c}+i\,{\sf
d}}_{Z_0(w)} +\underbrace{({\sf a}+i\,{\sf b})}_{Z_1(w)}z,
\\
W^1(z,w)&:=\underbrace{{\sf c}_1+2\,{\sf
a}\,w^1}_{W^1_0(w)}+2\,(\underbrace{{\sf d}+i\,{\sf
c}}_{W^1_1(w)})\,z,
\\
 W^2(z,w)&:=\underbrace{{\sf c}_2+4\,{\sf
c}\,w^1+3\,{\sf a}\,w^2-{\sf b}\,w^3}_{W^2_0(w)}+\underbrace{2\,({\sf
d}+i\,{\sf c})}_{W^2_2(w)}\,z^2,
\\
 W^3(z,w)&:=\underbrace{{\sf c}_3+4\,{\sf
d}\,w^1+{\sf b}\,w^2+3\,{\sf a}\,w^3}_{W^3_0(w)}+\underbrace{2\,({\sf
c}-i\,{\sf d})}_{W^3_2(w)}\,z^2,
\\
 W^4(z,w)&:=\underbrace{{\sf
c}_4+3\,{\sf c}\,w^2-3\,{\sf d}\,w^3+4\,{\sf a}\,w^4-2\,{\sf
b}\,w^5}_{W^4_0(w)}+\underbrace{2\,({\sf d}+i\,{\sf
c})}_{W^4_3(w)}\,z^3,
\\
W^5(z,w)&:=\underbrace{{\sf c}_5+3\,{\sf d}\,w^2+3\,{\sf
c}\,w^3+2\,{\sf b}\,w^4+4\,{\sf
a}\,w^5}_{W^5_0(w)}+\underbrace{2\,({\sf c}-i\,{\sf
d})}_{W^5_3(w)}\,z^3,
\endaligned
\end{equation*}
\begin{equation*}
\footnotesize\aligned
 W^6(z,w)&:=\underbrace{{\sf c}_6+2\,{\sf c}\,w^2+2\,{\sf
d}\,w^3+4\,{\sf a}\,w^6}_{W^6_0(w)},
\\
 W^7(z,w)&:=\underbrace{{\sf c}_7+2\,{\sf c}\,w^4+2\,{\sf
d}\,w^5+6\,{\sf c}\,w^6+5\,{\sf a}\,w^7-{\sf b}\,w^8}_{W^7_0(w)},
\\
W^8(z,w)&:=\underbrace{{\sf c}_8-2\,{\sf d}\,w^4+2\,{\sf
c}\,w^5+6\,{\sf d}\,w^6+{\sf b}\,w^7+5\,{\sf a}\,w^8}_{W^8_0(w)},
\endaligned
\end{equation*}
  and is generated by the twelve $\mathbb R$-linearly
independent holomorphic vector fields:
\begin{equation*}
\footnotesize \left\{ \aligned
 {\sf X}_i&:=\partial_{w_i}, \ \ \ \ \ {
i=1,\ldots,8},
\\
 {\sf X}_9&:=z\partial_{z}+2\,w^1\partial_{w^1}+3\,w^2\partial_{w^2}+3
\,w^3\partial_{w^3}+4\,w^4\partial_{w^4}+
 \\
&+
 4\,w^5\partial_{w^5}+4\,w^6\partial_{w^6}+5\,w^7\partial_{w^7}+5\,w^8
\partial_{w^8},
 \\
 {\sf
 X}_{10}&:=i\,z\partial_{z}-w^3\partial_{w^2}+w^2\partial_{w^3}-2\,w^5
\partial_{w^4}+2\,w^4\partial_{w^5}-
 w^8\partial_{w^7}+w^7\partial_{w^8},
 \\
 {\sf
 X}_{11}&:=\partial_z+2\,i\,z\partial_{w^1}+(4\,w^1+2\,i
\,z^2)\partial_{w^2}+2\,z^2\partial_{w^3}+(3\,w^2+2\,i
\,z^3)\partial_{w^4}+
 \\
 &+
 (3\,w^3+2\,z^3)\partial_{w^5}+2\,w^2\partial_{w^6}+(2\,w^4+6
\,w^6)\partial_{w^7}+2\,w^5\partial_{w^8},
 \\
 {\sf
 X}_{12}&:=i\,\partial_{z}+2\,z\partial_{w^1}+2\,z^2\partial_{w^2}+(4
\,w^1-2\,i\,z^2)\partial_{w^3}+(-3\,w^3+2\,z^3)\partial_{w^4}+
 \\
 &+
 (3\,w^2-2\,i\,z^3)\partial_{w^5}+2\,w^3\partial_{w^6}+2\,w^5
\partial_{w^7}+(-2\,w^4+6\,w^6)\partial_{w^8}.
 \endaligned
 \right.
\end{equation*}
Furthermore, it is graded of the form:
\[
\frak{aut}_{CR}(\mathbb M^1):=\frak g_{-5}\oplus\frak
g_{-4}\oplus\frak g_{-3}\oplus\frak g_{-2}\oplus\frak
g_{-1}\oplus\frak g_0
\]
with $\frak g_{-5}:=\langle{\sf X}_7,{\sf X}_8\rangle$,with $\frak
g_{-4}:=\langle{\sf X}_4,{\sf X}_5,{\sf X}_6\rangle$ ,with $\frak
g_{-3}:=\langle{\sf X}_2,{\sf X}_3\rangle$, with $\frak
g_{-2}:=\langle{\sf X}_1\rangle$, with $\frak g_{-1}:=\langle{\sf
X}_{11},{\sf X}_{12}\rangle$, and with $\frak g_0:=\langle{\sf
X}_9,{\sf X}_{10}\rangle$ and together with the following table of
commutators:
\begin{center}
\begin{tabular}{|c|c|c|c|c|c|c|c|c|c|c|c|c|}
  \hline
   & ${\sf X}_{1}$ & ${\sf X}_{2}$ & ${\sf X}_{3}$ & ${\sf X}_{4}$ &
${\sf X}_{5}$ & ${\sf X}_{6}$ & ${\sf X}_{7}$ &
    ${\sf X}_{8}$ & ${\sf X}_{9}$ & ${\sf X}_{10}$ & ${\sf X}_{11}$ &
${\sf X}_{12}$ \\
  \hline
  ${\sf X}_{1}$ & $0$ & $0$ & $0$ & $0$ & $0$ & $0$ & $0$ & $0$ & $2{\sf
X}_{1}$ & $0$ & $4{\sf X}_{2}$ & $4{\sf X}_{3}$ \\
  ${\sf X}_{2}$ & $\ast$ & $0$ & $0$ & $0$ & $0$ & $0$ & $0$ & $0$ &
$3{\sf X}_{2}$ & ${\sf X}_{3}$ & $3{\sf X}_{4}$+$2{\sf X}_{6}$ & $3{\sf
X}_{5}$ \\
  ${\sf X}_{3}$ & $\ast$ & $\ast$ & $0$ & $0$ & $0$ & $0$ & $0$ & $0$ &
$3{\sf X}_{3}$ & -${\sf X}_{2}$ & $3{\sf X}_{5}$ & $-3{\sf
X}_{4}$+$2{\sf X}_{6}$ \\
  ${\sf X}_{4}$ & $\ast$ & $\ast$ & $\ast$ & $0$ & $0$ & $0$ & $0$ & $0$
& $4{\sf X}_{4}$ & $2{\sf X}_{5}$ & $2{\sf X}_{7}$ & $-2{\sf X}_{8}$ \\
  ${\sf X}_{5}$ & $\ast$ & $\ast$ & $\ast$ & $\ast$ & $0$ & $0$ & $0$ &
$0$ & $4{\sf X}_{5}$ & $-2{\sf X}_{4}$ & $2{\sf X}_{8}$ & $2{\sf X}_{7}$
\\
  ${\sf X}_{6}$ & $\ast$ & $\ast$ & $\ast$ & $\ast$ & $\ast$ & $0$ & $0$
& $0$ & $4{\sf X}_{6}$ & $0$ & $6{\sf X}_7$ & $0$ \\
  ${\sf X}_{7}$ & $\ast$ & $\ast$ & $\ast$ & $\ast$ & $\ast$ & $\ast$ &
$0$ & $0$ & $5{\sf X}_{7}$ & ${\sf X}_{8}$ & $0$ & $0$ \\
  ${\sf X}_{8}$ & $\ast$ & $\ast$ & $\ast$ & $\ast$ & $\ast$ & $\ast$ &
$\ast$ & $0$ & $5{\sf X}_{8}$ & $-{\sf X}_{7}$ & $0$ & $0$ \\
  ${\sf X}_{9}$ & $\ast$ & $\ast$ & $\ast$ & $\ast$ & $\ast$ & $\ast$ &
$\ast$ & $\ast$ & $0$ & $0$ & -${\sf X}_{11}$ & -${\sf X}_{12}$ \\
  ${\sf X}_{10}$ & $\ast$ & $\ast$ & $\ast$ & $\ast$ & $\ast$ & $\ast$ &
$\ast$ & $\ast$ & $\ast$ & $0$ & -${\sf X}_{12}$ & ${\sf X}_{11}$ \\
  ${\sf X}_{11}$ & $\ast$ & $\ast$ & $\ast$ & $\ast$ & $\ast$ & $\ast$ &
$\ast$ & $\ast$ & $\ast$ & $\ast$ & $0$ & $4{\sf X}_1$ \\
  ${\sf X}_{12}$ & $\ast$ & $\ast$ & $\ast$ & $\ast$ & $\ast$ & $\ast$ &
$\ast$ & $\ast$ & $\ast$ & $\ast$ & $\ast$ & $0$ \\
  \hline
\end{tabular}
\end{center}
\end{Theorem}

\begin{Theorem}
\label{Theorem-3} The Lie algebra of infinitesimal CR-automorphisms
$\frak{aut}_{CR}(\mathbb M^3)$ of the rigid real analytic
CR-generic submanifold
$\mathbb M^3\subset\mathbb M^{1+8}$, represented as the
graph of the eight defining equations:
\[\aligned
w^j-\overline{w}^j&=\Xi_j(z,\overline{z}), \ \ \ \ \ \ \ \ \ \ \ \
{\scriptstyle (j\,=\,1\cdots\,6)},
\\
w^7-\overline{w}^7&=\Xi_7'(z,\overline{z}),
\\
w^8-\overline{w}^8&=\Xi_8''(z,\overline{z}):=2\,i\,(z^4\overline{z}+z
\overline{z}^4),
\endaligned
\]
 is $11$-dimensional with the coefficients:
\begin{equation*}
 \footnotesize\aligned
 Z(z,w)&:=\underbrace{{\sf c}+i\,{\sf
d}}_{Z_0(w)} +\underbrace{{\sf a}}_{Z_1(w)}z,
\\
W^1(z,w)&:=\underbrace{{\sf c}_1+2\,{\sf
a}\,w^1}_{W^1_0(w)}+2\,(\underbrace{{\sf d}+i\,{\sf
c}}_{W^1_1(w)})\,z,
\\
 W^2(z,w)&:=\underbrace{{\sf c}_2+4\,{\sf
c}\,w^1+3\,{\sf a}\,w^2}_{W^2_0(w)}+\underbrace{2\,({\sf d}+i\,{\sf
c})}_{W^2_2(w)}\,z^2,
\\
 W^3(z,w)&:=\underbrace{{\sf c}_3+4\,{\sf
d}\,w^1+3\,{\sf a}\,w^3}_{W^3_0(w)}+\underbrace{2\,({\sf c}-i\,{\sf
d})}_{W^3_2(w)}\,z^2,
\\
 W^4(z,w)&:=\underbrace{{\sf
c}_4+3\,{\sf c}\,w^2-3\,{\sf d}\,w^3+4\,{\sf
a}\,w^4}_{W^4_0(w)}+\underbrace{2\,({\sf d}+i\,{\sf
c})}_{W^4_3(w)}\,z^3,
\\
 W^5(z,w)&:=\underbrace{{\sf c}_5+3\,{\sf
d}\,w^2+3\,{\sf c}\,w^3+4\,{\sf
a}\,w^5}_{W^5_0(w)}+\underbrace{2\,({\sf c}-i\,{\sf
d})}_{W^5_3(w)}\,z^3,
\\
 W^6(z,w)&:=\underbrace{{\sf c}_6+2\,{\sf
c}\,w^2+2\,{\sf d}\,w^3+4\,{\sf a}\,w^6}_{W^6_0(w)},
\\
W^7(z,w)&:=\underbrace{{\sf c}_7+2\,{\sf c}\,w^4+2\,{\sf
d}\,w^5+6\,{\sf c}\,w^6+5\,{\sf a}\,w^7}_{W^7_0(w)},
\\
W^8(z,w)&:=\underbrace{{\sf c}_8+4\,{\sf c}\,w^4-4\,{\sf
d}\,w^5+5\,{\sf a}\,w^8}_{W^8_0(w)}+\underbrace{2\,({\sf d}+i\,{\sf
c})}_{W^8_4(w)}\,z^4,
\endaligned
\end{equation*}
  and is generated by the eleven $\mathbb R$-linearly
independent holomorphic vector fields:
\begin{equation*}
\footnotesize \left\{ \aligned
 {\sf X}_i&:=\partial_{w_i}, \ \ \ \ \ {
i=1,\ldots,8},
\\
 {\sf X}_9&:=z\partial_{z}+2\,w^1\partial_{w^1}+3\,w^2\partial_{w^2}+3
\,w^3\partial_{w^3}+4\,w^4\partial_{w^4}+
 \\
&+
 4\,w^5\partial_{w^5}+4\,w^6\partial_{w^6}+5\,w^7\partial_{w^7}+5\,w^8
\partial_{w^8},
 \\
 {\sf
 X}_{10}&:=\partial_z+2\,i\,z\partial_{w^1}+(4\,w^1+2\,i
\,z^2)\partial_{w^2}+2\,z^2\partial_{w^3}+(3\,w^2+2\,i
\,z^3)\partial_{w^4}+
 \\
 &+
 (3\,w^3+2\,z^3)\partial_{w^5}+2\,w^2\partial_{w^6}+(2\,w^4+6
\,w^6)\partial_{w^7}+(4\,w^4+2\,i\,z^4)\partial_{w^8},
 \\
 {\sf
 X}_{11}&:=i\,\partial_{z}+2\,z\partial_{w^1}+2\,z^2\partial_{w^2}+(4
\,w^1-2\,i\,z^2)\partial_{w^3}+(-3\,w^3+2\,z^3)\partial_{w^4}+
 \\
 &+
 (3\,w^2-2\,i\,z^3)\partial_{w^5}+2\,w^3\partial_{w^6}+2\,w^5
\partial_{w^7}+(-4\,w^5+2\,z^4)\partial_{w^8}.
 \endaligned
 \right.
\end{equation*}
Furthermore, it is graded of the form:
\[
\frak{aut}_{CR}(\mathbb M^1):=\frak g_{-5}\oplus\frak
g_{-4}\oplus\frak g_{-3}\oplus\frak g_{-2}\oplus\frak
g_{-1}\oplus\frak g_0
\]
with $\frak g_{-5}:=\langle{\sf X}_7,{\sf X}_8\rangle$,with $\frak
g_{-4}:=\langle{\sf X}_4,{\sf X}_5,{\sf X}_6\rangle$ ,with $\frak
g_{-3}:=\langle{\sf X}_2,{\sf X}_3\rangle$, with $\frak
g_{-2}:=\langle{\sf X}_1\rangle$, with $\frak g_{-1}:=\langle{\sf
X}_{10},{\sf X}_{11}\rangle$, and with $\frak g_0:=\langle{\sf
X}_9\rangle$ and together with the following table of commutators:

\medskip
\begin{center}
\begin{tabular}{|c|c|c|c|c|c|c|c|c|c|c|c|}
  \hline
   & ${\sf X}_{1}$ & ${\sf X}_{2}$ & ${\sf X}_{3}$ & ${\sf X}_{4}$ &
${\sf X}_{5}$ & ${\sf X}_{6}$ & ${\sf X}_{7}$ &
    ${\sf X}_{8}$ & ${\sf X}_{9}$ & ${\sf X}_{10}$ & ${\sf X}_{11}$  \\
  \hline
  ${\sf X}_{1}$ & $0$ & $0$ & $0$ & $0$ & $0$ & $0$ & $0$ & $0$ & $2{\sf
X}_{1}$ & $4{\sf X}_{2}$ & $4{\sf X}_{3}$ \\
  ${\sf X}_{2}$ & $\ast$ & $0$ & $0$ & $0$ & $0$ & $0$ & $0$ & $0$ &
$3{\sf X}_{2}$ & $3{\sf X}_{4}$+$2{\sf X}_{6}$ & $3{\sf X}_{5}$  \\
  ${\sf X}_{3}$ & $\ast$ & $\ast$ & $0$ & $0$ & $0$ & $0$ & $0$ & $0$ &
$3{\sf X}_{3}$ & $3{\sf X}_{5}$ & $-3{\sf X}_{4}$+$2{\sf X}_{6}$ \\
  ${\sf X}_{4}$ & $\ast$ & $\ast$ & $\ast$ & $0$ & $0$ & $0$ & $0$ & $0$
& $4{\sf X}_{4}$ & $2{\sf X}_{7}$+$4{\sf X}_{8}$ & $0$ \\
  ${\sf X}_{5}$ & $\ast$ & $\ast$ & $\ast$ & $\ast$ & $0$ & $0$ & $0$ &
$0$ & $4{\sf X}_{5}$ & $0$ & $2{\sf X}_{7}$-$4{\sf X}_{8}$ \\
  ${\sf X}_{6}$ & $\ast$ & $\ast$ & $\ast$ & $\ast$ & $\ast$ & $0$ & $0$
& $0$ & $4{\sf X}_{6}$ & $6{\sf X}_{7}$ & $0$ \\
  ${\sf X}_{7}$ & $\ast$ & $\ast$ & $\ast$ & $\ast$ & $\ast$ & $\ast$ &
$0$ & $0$ & $5{\sf X}_{7}$ & $0$ & $0$ \\
  ${\sf X}_{8}$ & $\ast$ & $\ast$ & $\ast$ & $\ast$ & $\ast$ & $\ast$ &
$\ast$ & $0$ & $5{\sf X}_{8}$ & $0$ & $0$ \\
  ${\sf X}_{9}$ & $\ast$ & $\ast$ & $\ast$ & $\ast$ & $\ast$ & $\ast$ &
$\ast$ & $\ast$ & $0$ & -${\sf X}_{10}$ & -${\sf X}_{11}$ \\
  ${\sf X}_{10}$ & $\ast$ & $\ast$ & $\ast$ & $\ast$ & $\ast$ & $\ast$ &
$\ast$ & $\ast$ & $\ast$ & $0$ & $4{\sf X}_{1}$ \\
  ${\sf X}_{11}$ & $\ast$ & $\ast$ & $\ast$ & $\ast$ & $\ast$ & $\ast$ &
$\ast$ & $\ast$ & $\ast$ & $\ast$ & $0$ \\
    \hline
\end{tabular}
\end{center}
\end{Theorem}

\end{document}